\documentclass{amsart}

\usepackage{amssymb, bbm} 
\usepackage{graphics}
\usepackage{color}
\usepackage{natbib}
\usepackage[a4paper, margin=3cm]{geometry}

\setcitestyle{open={[},close={]},comma,numbers}

\newcommand\N{{\mathbb{N}}}
\newcommand\R{{\mathbb{R}}}

\newcommand{\maybed}{}  
\newcommand{\maybedsigma}{S} 

\newcommand{\differential}{\mathrm{d}}
\newcommand{\loc}{\mathrm{loc}}
\newcommand{\Index}{\mathcal{I}}
\newcommand{\constantB}{{B}}  
\newcommand{\constantM}{{M}}
\newcommand{\constantY}{{Y}}
\newcommand{\constantC}{\mathcal{C}}

\newcommand{\simeqone}{>0}

\newcommand{\simeqepsilonone}{}

\newcommand{\Span}{\operatorname{span}}
\newcommand{\1}{\mathbbm{1}}
\newcommand\BL{\operatorname{BL}}
\newcommand\BLg{\operatorname{BL}_{\mathrm{g}}}
\newcommand{\datum}{\mathbf}

\newcommand{\gap}{\beta}
\newcommand{\sump}{\rho}
\newcommand{\loss}{\sigma}
\newcommand{\cutscl}{\delta}
\newcommand{\gausscl}{\delta^{1+\gamma}}

\let\oldepsilon\epsilon
\let\oldvarepsilon\varepsilon
\renewcommand{\epsilon}{\oldvarepsilon}
\renewcommand{\varepsilon}{\oldepsilon}

\theoremstyle{plain}
  \newtheorem{theorem}{Theorem}[section]

  \newtheorem{proposition}[theorem]{Proposition}
  \newtheorem{prop}[theorem]{Proposition}
  
  \newtheorem{lemma}[theorem]{Lemma}
  \newtheorem{corollary}[theorem]{Corollary}

\theoremstyle{definition}
  \newtheorem*{remark*}{Remark}
  \newtheorem{definition}[theorem]{Definition}
  \newtheorem{example}[theorem]{Example}
  
  \newtheorem*{question*}{Question}
  \newtheorem*{notation}{Notation}

\numberwithin{equation}{section}
\date{\today}
\title{On the nonlinear Brascamp--Lieb inequality}

\author[J. Bennett]{Jonathan Bennett}
\address{School of Mathematics, The Watson Building, University of Birmingham, Edgbaston,
Birmingham, B15 2TT, England.}
\email{J.Bennett@bham.ac.uk}

\author[N. Bez]{Neal Bez}
\address{Department of Mathematics, Graduate School of Science and Engineering,
Saitama University, Saitama 338-8570, Japan.}
\email{nealbez@mail.saitama-u.ac.jp}

\author[S. Buschenhenke]{Stefan Buschenhenke}
\address{Mathematisches Seminar, Christian-Albrechts-Universit\"at zu Kiel, 24116 Kiel, Germany.}
\email{buschenhenke@math.uni-kiel.de}

\author[M. G. Cowling]{Michael G. Cowling}
\address{School of Mathematics and Statistics, University of New South
Wales, Sydney NSW 2052, Australia.}
\email{m.cowling@unsw.edu.au}

\author[T. C. Flock]{Taryn C. Flock}
\address{Department of Mathematics and Statistics,
Lederle Graduate Research Tower, University of Massachusetts, 710 N.
Pleasant Street,
Amherst, MA 01003-9305, USA.} \curraddr{Department of Mathematics, Statistics, and Computer Science, Macalester College, 1600 Grand Ave, St Paul MN 55105 USA.}
\email{tflock@macalester.edu}

\thanks{The first, third and fifth authors were supported by the European Research Council [grant
number 307617], the second author by JSPS Grant-in-Aid for Young Scientists A [grant number 16H05995] and JSPS Grant-in-Aid for Challenging Exploratory Research [grant number 16K13771-01], and the fourth author by the Australian Research Council [grant number DP170103025]. }

\begin{document}
\begin{abstract}
We prove a nonlinear variant of the general Brascamp--Lieb inequality.
Our proof consists of running an efficient, or ``tight", induction on scales argument, which uses the existence of gaussian near-extremisers to the underlying linear Brascamp--Lieb inequality (Lieb's theorem) in a fundamental way.
A key ingredient is an effective version of Lieb's theorem, which we establish via a careful analysis of near-minimisers of weighted sums of exponential functions.
Instances of this inequality are quite prevalent in mathematics, and we illustrate this with some applications in harmonic analysis.
\end{abstract}
\maketitle

\section{Introduction}
The Brascamp--Lieb inequality is a well-known and far-reaching generalisation of a wide range of sharp functional inequalities in analysis, including the multilinear H\"older, Loomis--Whitney and Young convolution inequalities.
It takes the form
\begin{equation}\label{BL}
\int_{\R^n}\prod_{j=1}^mf_j^{p_j}(L_jx) \,dx
\leq \BL(\datum{L},\datum{p})\prod_{j=1}^m\left(\int_{\R^{n_j}}f_j(x_j) \,dx_j\right)^{p_j}
\end{equation}
over all nonnegative integrable functions $f_j$, where the mappings $L_j:\R^n\to\R^{n_j}$ are linear surjections, the exponents $p_j\in [0,1]$, and $\BL(\datum{L},\datum{p})$ denotes the smallest constant (which may be infinite).
We refer to $(\datum{L},\datum{p})=((L_j)_{j=1}^m,(p_j)_{j=1}^m)$ as the \textit{Brascamp--Lieb datum}, and $\BL(\datum{L},\datum{p})$ as the \textit{Brascamp--Lieb constant}.
This celebrated inequality was first formulated in \cite{BL} and has given rise to an elegant theory with wide-ranging applications across the mathematical sciences, from convex geometry, to information theory, harmonic analysis, partial differential equations, number theory and theoretical computer science. We refer to \cite{BCCT1}, \cite{Ball}, \cite{Barthe}, \cite{CC}, \cite{B}, \cite{Brown}, \cite{BDGuth}, \cite{GGOW} and the references there for some further discussion of these and other connections.
A natural question that has arisen in harmonic analysis and partial differential equations in recent years, beginning with \cite{BCW}, is whether one may replace the linear surjections $L_j$ in \eqref{BL} by smooth submersions, at least in a neighbourhood of a point.
The main purpose of this paper is to give an affirmative answer to this question.
\begin{theorem}\label{main}
Let $(\datum{L},\datum{p})$ be a Brascamp--Lieb datum.
Suppose that  $B_j:\R^n\to\R^{n_j}$ are $C^2$ submersions in a neighbourhood of a point $x_0$ and $dB_j(x_0)=L_j$, where $1\leq j\leq m$.
Then, for each $\epsilon>0$ there exists a neighbourhood $U$ of $x_0$ such that
\begin{equation}\label{nlbli}
\int_U \prod_{j=1}^m f_j^{p_j}(B_j(x)) \,dx
\leq (1+\epsilon)\BL(\datum{L},\datum{p})\prod_{j=1}^m\left(\int_{\R^{n_j}}f_j(x_j) \,dx_j \right)^{p_j}.
\end{equation}
\end{theorem}
Theorem \ref{main} has some significant applications, for which we refer to Section~\ref{Sect:Applications}.

As may be expected, Theorem \ref{main} is equivalent to a superficially stronger statement where the underlying euclidean spaces are replaced by smooth manifolds.
This is a reflection of the (essential) diffeomorphism-invariance of its statement.
While this is arguably the natural context for such nonlinear Brascamp--Lieb inequalities, we choose to work in the euclidean setting for concreteness.

Theorem \ref{main} was previously unknown with any finite constant in place of the optimal $(1+\epsilon)$ factor,
and we do not know how to prove such a statement without the full argument of this paper.
Prior to Theorem \ref{main}, such nonlinear statements were only known under additional structural constraints on the nonlinear maps $\datum{B}$.
This began with \cite{BCW} and \cite{BB}, where the underlying linear maps $\datum{L}$ in the statement of Theorem \ref{main} were required to satisfy the highly-restrictive ``basis condition''
\begin{equation}\label{bas}
\bigoplus _{j=1}^m \ker L_j=\R^n,
\end{equation}
which keeps the complexity at the level of the classical Loomis--Whitney inequality; see also \cite{BHT}, \cite{KS}, \cite{CHV}, and applications in \cite{BHHT}, \cite{BH}, \cite{FGH}, \cite{HK}, \cite{HeK}, and \cite{K}.
We note that the results in \cite{CHV} also feature the best-possible $(1+\epsilon)$ factor present in \eqref{nlbli}, and cater for a wider class of data originating in work of Finner \cite{Finner}.
Further examples beyond the scope of \eqref{bas} may also be found in unpublished work of Stovall.
Very different arguments in \cite{Z} allowed for polynomial $\datum{B}$, but generated bounds that depend strongly on the degree.
The only nonlinear statements that were known for completely general data involved an $\epsilon$-loss in the scale of Sobolev spaces; see \cite{BBFL}.

Theorem \ref{main} has proved quite elusive.
This is perhaps to be expected, as the analysis of the classical inequality \eqref{BL} is delicate and relies heavily on linear-algebraic arguments throughout.
A particularly important example of linearity at work is in a fundamental theorem of Lieb \cite{L}, which states that the Brascamp--Lieb constant $\BL(\datum{L},\datum{p})$ is determined by centred gaussian inputs $f_j$; that is
\begin{equation}\label{gau}
f_j(x)=c_j\exp(-\pi\langle A_jx,x\rangle),
\end{equation}
where $A_j$ is a positive definite $n_j\times n_j$ matrix, and $c_j$ is a positive constant.
For such inputs,
\[
\frac{\int_{\R^n}\prod_{j=1}^m f_j^{p_j}(L_jx) \,dx}
        {\prod_{j=1}^m \left(\int_{\R^{n_j}}f_j(x_j) \,dx_j \right)^{p_j}}
=\frac{\prod_{j=1}^m\det(A_j)^{p_j/2}}{\det\left(\sum_{j=1}^mp_jL_j^*A_jL_j\right)^{1/2}},
\]
and so Lieb's theorem may be stated as follows:
\begin{theorem}[Lieb]\label{LIEB}
\begin{equation}\label{Lie}
\BL(\datum{L},\datum{p})=\sup_{\datum{A}}\BLg (\datum{L},\datum{p};\datum{A}),
\end{equation}
where $\datum{A}=(A_j)_{j=1}^m$, and
\[
\BLg (\datum{L},\datum{p};\datum{A})
=\frac{\prod_{j=1}^m\det(A_j)^{p_j/2}}{\det\left(\sum_{j=1}^mp_jL_j^*A_jL_j\right)^{1/2}}.
\]
\end{theorem}
Lieb's theorem, which clearly only makes sense for linear $L_j$, has been central to much of our understanding of the Brascamp--Lieb inequality.
For instance, it was used in \cite{BCCT1} to
show that $\BL(\datum{L},\datum{p})$ is finite if and only if the transversality condition
\begin{equation}\label{char}
\dim(V)\leq\sum_{j=1}^m p_j \dim(L_jV)
\end{equation}
holds for all subspaces $V$ of $\R^n$, along with the scaling condition
\begin{equation}\label{scal}
n=\sum_{j=1}^m p_jn_j.
\end{equation}
In the setting of \textit{nonlinear} data, three viable approaches have emerged: (i) the method of \textit{refinements} of Christ \cite{Ch}, (ii) the \textit{factorisation} method of Carbery, H\"anninen and Valdimarsson \cite{CHV}, and (iii) the method of \textit{induction-on-scales}, introduced in this setting by Bejenaru, Herr and Tataru \cite{BHT}.
All of these methods are at least effective under the rather rigid Loomis--Whitney-type structural condition \eqref{bas}, as may be seen in \cite{BCW}, \cite{CHV} and \cite{BB} respectively.

Our approach is based on the method of induction-on-scales.
This general method has proved to be extraordinarily effective in harmonic analysis since its introduction by Bourgain \cite{Bo} in the 1990s.
The basic idea is to introduce an auxiliary scale to the inequality in hand --- in our case, the diameter of the neighbourhood $U$ appearing in the statement of Theorem \ref{main} --- and show that for some suitable increasing sequence of scales, the relevant bounds are suitably stable as one passes from one scale to the next.
Usually there are many inefficiencies in passing between scales, and losses are incurred.
As a result the method rarely leads to endpoint estimates, although there are some notable exceptions, such as \cite{TaoBil} and \cite{BHT}.
Our proof of Theorem \ref{main} is another such exception, also requiring us to find an efficient passage between scales.
This turns out to be rather delicate, and in particular, requires us to find a suitably effective version of Lieb's theorem (Theorem \ref{LIEB}).
This ultimately requires us to address some natural questions in optimisation theory, which appear to be largely absent in the literature, and have some independent interest.
Our effective version --- the forthcoming Theorem \ref{efflie} --- is phrased in terms of ``$\delta$-near-extremisers", also known as ``$\delta$-optimal solutions" in the optimisation theory literature (see, for example, \cite{BS}).
We refer to an input $\datum{f}=(f_j)_{j=1}^m$ as a \textit{$\delta$-near-extremiser} for \eqref{BL} if
\begin{equation} \label{delnear}
\BL(\datum{L}, \datum{p}; \datum{f})
:=\frac{\int_{\R^n}\prod_{j=1}^m f_j^{p_j}(L_jx) \,dx} {\prod_{j=1}^m\left(\int_{\R^{n_j}}f_j(x_j) \,dx_j\right)^{p_j}}
\geq(1-\delta)\BL(\datum{L},\datum{p}).
\end{equation}
Of course, Lieb's Theorem guarantees that every Brascamp--Lieb functional $\datum{f} \mapsto \BL(\datum{L}, \datum{p}; \datum{f})$ has \textit{gaussian} $\delta$-near-extremisers, for all $\delta>0$.
What it doesn't tell us is anything quantitative about the set of gaussian $\delta$-near-extremisers, and this is what our effective version of Lieb's theorem will do.
In order to state this quantified form of Lieb's theorem, we shall fix $\datum{p}$ and underlying dimensions $n,n_1,\dots,n_m$, and refer to an $m$-tuple of linear maps $\datum{L}$ as \textit{admissible} if each $L_j:\R^n\to\R^{n_j}$ is a surjection (in other words, $(\datum{L},\datum{p})$ is a Brascamp--Lieb datum). 
\begin{theorem}[Effective version of Lieb's theorem]\label{efflie}
There exist $N\in \N$ and $\delta_0>0$ such that for every $\delta \in (0, \delta_0)$ and admissible $\datum{L}$,
\begin{equation} \label{effin}
 \inf_{\substack{\|C_j^*A_jC_j\|\leq \delta^{-N}\\ \|(C_j^*A_jC_j)^{-1}\|\leq \delta^{-N}}}
\BLg (\datum{L},\datum{p};\datum{A})^{-2} \leq
\BL(\datum{L}, \datum{p})^{-2}
+\delta\prod_{j=1}^m\det(L_jL_j^*)^{p_j}.
\end{equation}
Here $C_j=(L_jL_j^*)^{1/2}$.
\end{theorem}
We remark that whilst $\BLg (\datum{L},\datum{p};\datum{A})$ is invariant under the isotropic rescaling $\datum{A} \mapsto \lambda \datum{A}$, the conditions that $\|C_j^*A_jC_j\|\leq \delta^{-N}$ and $\|(C_j^*A_jC_j)^{-1}\|\leq \delta^{-N}$ are not (this point will re-emerge in the course of our proof of Theorem \ref{main}).
Also, on first sight it is perhaps a little unclear that Theorem \ref{efflie} is a statement about near-extremisers.
However, it is straightforward to conclude the following.
\begin{corollary}\label{effliecor} Suppose that  $(\datum{L}_0,\datum{p})$ is a Brascamp--Lieb datum with finite Brascamp--Lieb constant.
Then there exists $N\in\N$ and $\delta_0>0$ such that for every $\delta \in (0, \delta_0)$,
\[
\sup_{\substack{\|A_j\|,\|A_j^{-1}\|<\delta^{-N}}}  \BLg (\datum{L},\datum{p};\datum{A}) \geq(1-\delta)\BL(\datum{L},\datum{p})
\]
for all $\datum{L}$ sufficiently close to $\datum{L}_0$.
\end{corollary}

In brief, both of these statements tell us that we may find gaussian $\delta$-near-extremisers for the Brascamp--Lieb inequality, whose eccentricity grows at most polynomially in $1/\delta$, and that this growth rate can be taken to be locally uniform in the linear maps $\datum{L}$.
It is perhaps interesting to note that Theorem \ref{efflie} is considerably stronger that Corollary \ref{effliecor}; in particular, the former is \emph{universal} in the sense that the constants $N$ and $\delta_0$ are entirely independent of $\datum{L}$, given fixed $\datum{p}$ and $n,n_1,\dots,n_m$.

Theorem \ref{efflie} (or Corollary \ref{effliecor}) has the potential for application wherever Lieb's theorem is of use.
For instance, the \emph{continuity} of the classical Brascamp--Lieb constant $\BL(\cdot,\datum{p})$, established using Lieb's theorem in \cite{BBCF} and \cite{GGOW}, may easily be improved to \emph{local H\"older continuity} using Corollary \ref{effliecor} --- see Section \ref{Sect:5} for a short proof of this fact.

Our approach to Theorem \ref{efflie} is to reduce it to a similar statement involving ``$\delta$-near-minimisers" for functions of the form
\begin{equation}\label{weightgauss}
\R^n\ni y\mapsto\sum_{j\in J}d_j\exp(\langle  u_j,y \rangle),
\end{equation}
where each $d_j$ is a nonnegative real number (coefficient), and each $u_j$ is a vector in $\R^n$.
Here, we are referring to a vector $y_\delta$ as a \emph{$\delta$-near-minimiser} for the function in \eqref{weightgauss} if
\begin{equation*}
\sum_{j\in J}d_j\exp(\langle  u_j,y_\delta  \rangle)
\leq \inf_{y \in \R^n} \sum_{j\in J}d_j\exp(\langle  u_j,y \rangle) + \delta \max_{j \in J} d_j.
\end{equation*}
Notice that this is a near-minimiser in an ``additive" sense; this differs somewhat from our earlier definition of a $\delta$-near-extremiser in \eqref{delnear} for the Brascamp--Lieb functional.
An elementary example of a function of the form \eqref{weightgauss} is the hyperbolic cosine, which of course attains its minimum at the origin, although minimisers do not need to exist in general --- for instance if $J$ contains only one element.
Our main result in this setting is the following.
\begin{theorem}\label{effexp}
There exist $N\in \N$ and $\delta_0>0$, depending only on $(u_j)_{j\in J}$, such that for every $\delta \in (0, \delta_0)$ and nonnegative coefficients $d_j$,
\[
\inf_{|y|\leq N\log(1/\delta)}\sum_{j\in J}d_j\exp(\langle  u_j,y \rangle)\leq \inf_{y \in \R^n}\sum_{j\in J}d_j\exp(\langle  u_j,y \rangle)+\delta\max_{j\in J}d_j.
\]
\end{theorem}
An important feature of Theorem \ref{effexp} is the \textit{uniformity} of the conclusion in the coefficients $d_j$.
Theorem 1.5 tells us that we can always find a $\delta$-near-minimiser $y$ within a distance $N\log(1/\delta)$ of the origin, provided that $N$ and $1/\delta$ are sufficiently large (depending only on the $u_j$). We briefly describe a simple example that illustrates the nature and essential optimality of this result in Section \ref{Sect:5}.

\subsubsection*{Structure of the paper}
In Section \ref{Sect:Applications} we present our applications of the nonlinear Brascamp--Lieb inequality (Theorem \ref{main}) --- see the forthcoming Corollaries \ref{t:PDEapp} and \ref{LieCor}.
In Section \ref{Sect:3} we describe the strategy behind our proof of Theorem \ref{main}, with emphasis on the inductive process.
In Section \ref{Sect:4} we develop a near-optimisation theory of functions of the form \eqref{weightgauss}, proving Theorem \ref{effexp}.
In Section \ref{Sect:5} we deduce Theorem \ref{efflie} from Theorem \ref{effexp}, and in Section \ref{Sect:6} we prove Theorem \ref{main}.
\subsubsection*{Acknowledgments} The authors wish to thank Tony Carbery for many in-depth discussions on the subject of this paper.
They also thank Alessio Martini for discussions surrounding the local sharp version of Young's convolution inequality. Finally the authors thank the anonymous referees for their careful reading and valuable recommendations.

\section{Applications}\label{Sect:Applications}
From the point of view of applications, the Brascamp--Lieb inequality \eqref{BL} is often more usefully expressed as a bound on a multilinear form, namely
\begin{equation}\label{BLL}
\int_{\R^n}\prod_{j=1}^m f_j(L_jx) \,dx
\leq \BL(\datum{L},\datum{p})\prod_{j=1}^m\|f_j\|_{L^{q_j}(\R^{n_j})};
\end{equation}
here we have simply replaced $f_j^{p_j}$ with $f_j$, and $p_j$ with $1/q_j$.
This is a parametrised version of the more intrinsic inequality
\begin{equation}\label{stefanlikes}
\int_H f_1\otimes\dots\otimes f_m \,d\mu_H\leq \BL(H,\mathbf{q})\prod_{j=1}^m\|f_j\|_{L^{q_j}(\R^{n_j})},
\end{equation}
where $H$ is a subspace of $\mathbb{R}^{n_1}\times\cdots\times\mathbb{R}^{n_m}$, and by a minor abuse of notation, we use $\BL(H,\mathbf{q})$ to denote the best constant.
Here, integration in \eqref{stefanlikes} is with respect to Lebesgue measure on $H$, and the connection between \eqref{BLL} and \eqref{stefanlikes} of course comes from expressing $H$ as the range of the linear map $x \mapsto (L_1x,\dots,L_mx)$; we caution that the constants $\BL(H,\mathbf{q})$ and $\BL(\datum{L},\datum{p})$ will not in general coincide under this identification, although their finiteness is certainly equivalent.
From this perspective Theorem \ref{main} acquires the following useful form:
\begin{theorem}\label{main'}
Suppose $M$ is a $C^2$ submanifold of $\mathbb{R}^{n_1}\times\cdots\times\mathbb{R}^{n_m}$ and $x\in M$. Then given any $\epsilon>0$ there exists $\delta>0$ such that
\begin{equation}\label{great!}
\int_{M\cap B(x,\delta)}f_1\otimes\cdots\otimes f_m \,d\mu_M\leq (1+\epsilon)\BL(T_xM,\mathbf{q})\prod_{j=1}^m\|f_j\|_{L^{q_j}(\R^{n_j})}.
\end{equation}
Here $d\mu_M$ denotes Lebesgue measure on $M$, and $T_xM$ denotes the tangent space to $M$ at $x$.
\end{theorem}
Multilinear forms of the type handled by Theorem \ref{main'} arise often in analysis.
In this section we illustrate this with two substantial applications.
The first provides a broad family of multilinear convolution inequalities that arise naturally in dispersive PDE, where the generality of Theorem \ref{main'} allows for rather minimal structural hypotheses. The second application concerns best constants in Young's convolution inequality on an abstract Lie group, and capitalises on the optimality of the nonlinear Brascamp--Lieb constant in Theorem \ref{main}. 
\begin{remark*}
Theorem \ref{main'} belongs to the theory of Radon-like transforms, and provides a sharp multilinear estimate under minimal first order structural (as opposed to regularity) hypotheses on the manifold $M$.
It is not difficult to show that imposing additional higher order (curvature-related) conditions does not lead to improved bounds; see \cite[Lemma 6]{BBG} for an explicit statement of this fact.
However, it is important to point out that multilinear Radon-like transforms often arise with no underlying finite Brascamp--Lieb constant, yet estimates are possible under alternative (higher order) nondegeneracy conditions on $M$ --- see \cite{TW} and \cite{Stovall}.
Obtaining such a higher order result at the level of generality of Theorem \ref{main'} would be of considerable interest.
\end{remark*}

\subsection{Euclidean harmonic analysis and partial differential equations}\label{subsect:pde}
The regularising effect of convolution on functions, and more generally measures, has long been a matter at the heart of harmonic analysis and its applications. A particularly natural question that arises in both the restriction theory of the Fourier transform and nonlinear dispersive PDE, concerns convolutions of singular measures supported on submanifolds of $\mathbb{R}^n$, and seeks to understand the conditions under which such convolutions may give rise to bounded or $L^p$ densities -- see for example, \cite{TVV}, \cite{TaoAJM}.
Such phenomena have been important, for example, in recent breakthroughs in the low-regularity theory of the Zakharov system of \cite{BHHT} and \cite{BH}, and rest on certain ``transversality" properties of the families of submanifolds involved. In these works the underlying estimates ultimately rely on nonlinear perturbations of the Brascamp--Lieb inequality for data satisfying the Loomis--Whitney type condition \eqref{bas}. There are further manifestations of the nonlinear Brascamp--Lieb inequality in PDE-related problems -- see for example, \cite{FGH}, \cite{K}, \cite{HK}, \cite{HeK}, \cite{BBG2}.

In this section we use our general nonlinear Brascamp--Lieb inequality (in the form of Theorem \ref{main'}) to establish a much broader class of multilinear convolution estimates, where the manifolds involved satisfy minimal structural (or ``transversality") conditions.

For $1 \leq j \leq m$, we let $\Sigma_j : U_j \to \R^n$ be a parametrisation of a $C^2$, bounded and $n_j$-dimensional submanifold $S_j$ of $\R^{n}$.
We equip $S_j$ with the measure $d\maybed\sigma_j$, which is the push-forward of Lebesgue measure on $U_j$, that is,
\[
\int_{S_j} g(y) \,d\sigma_j(y) = \int_{U_j} g (\Sigma_j (x)) \,dx
\]
for any continuous function $g$ on $S_j$. For notational convenience, and without loss of generality, we may suppose that these surfaces are parametrised in such a way that $0\in U_j$ for each $j$, so that the image of the differential $d\Sigma_j(0)$ is the tangent space to $S_j$ at the (arbitrary) point $\Sigma_j(0)$.
Finally, we denote by $d\datum{\Sigma}(0)^*$ the $m$-tuple of linear maps $(d\Sigma_1(0)^*, \dots, d\Sigma_m(0)^*)$.
\begin{corollary} \label{t:PDEapp}
Suppose that  $\BL(d\datum{\Sigma}(0)^*,\datum{p}) < \infty$.
Then
\begin{equation} \label{e:multconv}
\|g_1d\maybed\sigma_1 * \dots * g_md\maybed\sigma_m\|_{L^\infty(\R^n)} \leq C \prod_{j=1}^m \|g_j\|_{L^{q_j}(\maybedsigma_j)}
\end{equation}
for all $g_j \in L^{q_j}(\maybedsigma_j)$ with sufficiently small support, where $q_j = (1/p_j)'=1/(1-p_j)$.
\end{corollary}
Some remarks are in order.
The first thing to observe is that the condition $\BL(d\datum{\Sigma}(0)^*,\datum{p}) < \infty$ may be replaced with the manifestly geometric condition
\begin{equation}\label{geoPDE}
\sum_{j=1}^m\frac{\dim(V\cap (T_{\Sigma_j(0)}S_j)^\perp)}{q_j'}\leq (\rho-1)\dim(V);
\end{equation}
here $\rho=\sum_{j=1}^m 1/q_j'$, and \eqref{geoPDE} should hold for all subspaces $V$ of $\mathbb{R}^n$, with equality when $V=\mathbb{R}^n$. This is simply an interpretation of the finiteness condition \eqref{char} in this context, and we clarify that $T_{\Sigma_j(0)}S_j$ denotes the tangent space to $S_j$ at $\Sigma_j(0)$.


Corollary \ref{t:PDEapp} is closely related to the conjectural ``restriction Brascamp--Lieb inequality" studied in \cite{BBFL} and \cite{Z}. This very general conjectural inequality states that, under the hypotheses of Corollary \ref{t:PDEapp},
\begin{equation}\label{resbl}
\int_{\mathbb{R}^n}\prod_{j=1}^m |\widehat{g_jd\sigma_j}(x)|^{2p_j} \, dx \leq C\prod_{j=1}^m\|g_j\|_{L^2(S_j)}^{2p_j};
\end{equation}
see \cite{BBFL} for a weaker local form of this inequality, which generalises the multilinear restriction inequality of \cite{BCT}, along with those used in the recent theory of decouplings in, for example, \cite{BDGuth} and \cite{BDGuo}.
For instance,  if $p_j\leq 1/2$ (or equivalently $1\leq q_j\leq 2$), for all $j$, then \eqref{resbl} very quickly implies \eqref{e:multconv}, at least when the $g_j$ are characteristic functions of sets. To see this, we observe that
\begin{eqnarray*}
\begin{aligned}
\|g_1d\sigma_1 * \cdots * g_md\sigma_m\|_{L^\infty(\mathbb{R}^n)}&\leq\int_{\mathbb{R}^n}\prod_{j=1}^m |\widehat{g_jd\sigma_j}(x)| \, dx\\
&\leq\prod_{j=1}^{m}\|\widehat{g_jd\sigma_j}\|_{L^\infty(\mathbb{R}^n)}^{1-2p_j}\int_{\mathbb{R}^n}\prod_{k=1}^m |\widehat{g_kd\sigma_k}(x)|^{2p_k}\, dx \\
&\leq C \prod_{j=1}^{m}\|g_j\|_{L^
1(S_j)}^{1-2p_j}\prod_{k=1}^m\|g_k\|_{L^
2(S_k)}^{2p_k}
= C\prod_{j=1}^m\|g_j\|_{L^
{q_j}(S_j)}.
\end{aligned}
\end{eqnarray*}
Here we have used the trivial $L^1\rightarrow L^\infty$ bound for the extension operators, along with the conjectural inequality \eqref{resbl}, and the assumption that the $g_j$ are characteristic functions of sets. Similar remarks in the special case of the Loomis--Whitney inequality were made in \cite{BHT} and \cite{BB}. We refer to \cite{BBRIMS} for a detailed analysis of the relationship between \eqref{e:multconv} and \eqref{resbl}.



Corollary \ref{t:PDEapp} contains subtleties which are related to the appearance of the dual exponent $q_j'$ in the transversality assumption \eqref{geoPDE}. In order to deal with these we first establish the following general (Fourier) duality principle for Brascamp--Lieb data, phrased in terms of the Brascamp--Lieb constant appearing in \eqref{stefanlikes}.

\begin{proposition} \label{p:FourierBL} Suppose $\datum{q} = (q_j)_{j=1}^m \in [1,\infty]^m$. Then
$\BL(H,\datum{q}) = \constantB_\datum{q} \BL(H^\perp,\datum{q'}),$
where
\[
\constantB_\datum{q} = \prod_{j=1}^m A_{q_j}^{n_j}
\]
and
\begin{equation} \label{e:BecknerC}
A_{r} = \bigg( \frac{r^{1/r}}{r'^{1/r'}} \bigg)^{1/2}.
\end{equation}
\end{proposition}
\begin{proof}
Since ${\constantB_{\datum{q}}^{-1}} = \constantB_{\datum{q'}}$, it suffices to prove the inequality
\begin{equation} \label{e:FourierBLineq}
\BL(H,\datum{q}) \leq \constantB_\datum{q} \BL(H^\perp,\datum{q'}).
\end{equation}
To see this, we use Lieb's theorem (Theorem \ref{LIEB}) to write
\[
\BL(H,\datum{q}) = \lim_{\nu \to 0} \BL(H,\datum{q};\datum{f}_{\nu}),
\]
where $f_{\nu,j}(x) =\exp(-\langle A_{\nu,j} {x} , {x}\rangle)$ for a suitable family of positive definite matrices $A_{\nu,j}$.
By a generalised form of Parseval's identity,
\begin{equation} \label{e:Parseval}
\int_H f_1 \otimes \dots \otimes f_m \,d\mu_H
= \int_{H^\perp} \widehat{f}_1 \otimes \dots \otimes \widehat{f}_m \,d\mu_{H^\perp}
\end{equation}
for all sufficiently nice $f_j$ (the integrals involve the Lebesgue measures on $H$ and $H^\perp$), and therefore it follows that
\[
\BL(H,\datum{q};\datum{f_{\nu}}) = \BL(H^\perp,\datum{q}';\datum{\widehat{f}_{\nu}}) \prod_{j=1}^m \frac{\|\widehat{f}_{\nu,j} \|_{q_j'}}{\|f_{\nu,j} \|_{q_j}}
\]
for each $\nu > 0$.
An explicit computation reveals that
\[
\prod_{j=1}^m \frac{\|\widehat{f}_{\nu,j} \|_{q_j'}}{\|f_{\nu,j} \|_{q_j}} = \constantB_\datum{q}
\]
for each $\nu > 0$, and thus \eqref{e:FourierBLineq} follows.
\end{proof}
Proposition \ref{p:FourierBL} appears to be new, and so merits some brief remarks. First of all it implies the ``Fourier--Brascamp--Lieb inequality"
$$
\int_H f_1\otimes\dots\otimes f_m \,d\mu_H\leq \BL(H,\mathbf{q})\prod_{j=1}^m A_{q_j}^{-n_j}\|\widehat{f}_j\|_{L^{q_j'}(\R^{n_j})},
$$
which may, depending on the size of the exponents $q_j$ relative to $2$, be a stronger bound than the usual Brascamp--Lieb inequality \eqref{BLL}. An early manifestation of this idea may be found in \cite{Ball}.

Derived from the involution $(H,\datum{q}) \mapsto (H^\perp,\datum{q'})$, there is an involution acting on Brascamp--Lieb data $(\datum{L},\datum{p})$, which gives rise to linear mappings $\widetilde{L}_j : \R^{\widetilde{n}} \to \R^{n_j}$ determined (not uniquely) by
\[
\sum_{j=1}^m L_j^*\widetilde{L}_j = 0;
\]
here $\widetilde{n}=n_1+\cdots+n_m-n$.
As a concrete instance of this, one may easily check that data associated with Young's convolution inequality is transformed to data associated with the trilinear H\"older inequality; the best constant in the latter case is equal to one and thus Proposition \ref{p:FourierBL} generates the best constant in Young's convolution inequality.
Note that we are not providing any kind of simplification to the derivation of the best constant in Young's convolution inequality since, in the course of our proof of Proposition \ref{p:FourierBL}, we made use of Lieb's theorem.
However, it provides an interesting perspective which may be fruitful in similar contexts where, for example, the data $\widetilde{\datum{L}}$ is rather more manageable than the data $\datum{L}$.
\begin{proof}[Proof of Corollary \ref{t:PDEapp}]
We begin by observing that
\begin{equation} \label{e:multconvatx}
g_1d\maybed\sigma_1 * \dots * g_md\maybed\sigma_m(x)
= \int_{U_1 \times \dots \times U_m}  f_1 \otimes \dots \otimes f_m(y) \delta(F(y,x))  \,dy,
\end{equation}
where $f_j = g_j \circ \Sigma_j$, and
\begin{equation} \label{e:convF}
F(y,x) = \Sigma_1(y_1) + \dots + \Sigma_m(y_m) - x.
\end{equation}
By considering input functions with sufficiently small supports, it suffices to prove that
\begin{equation*}
|g_1d\maybed\sigma_1 * \dots * g_md\maybed\sigma_m(0)| \leq C\prod_{j=1}^m \|g_j\|_{L^{q_j}(\maybedsigma_j)}.
\end{equation*}
Hence by \eqref{e:multconvatx}, it is enough to prove that
\begin{equation*}
\int_M f_1 \otimes \dots \otimes f_m \,d\mu
\leq C \prod_{j=1}^m \|f_j\|_{L^{q_j}(\R^{n_j})}
\end{equation*}
for $f_j$ supported in a sufficiently small neighbourhood of the origin, where
\[
M = \{ (y_1,\ldots,y_m) \in U_1 \times \dots \times U_m : F(y,0) = 0\},
\]
$\mu$ is a suitable measure on $M$, and $F$ is given by \eqref{e:convF}.
By Theorem \ref{main'} it suffices to prove that $\BL(T_0M,\datum{q}) < \infty$ where
$
T_0M
$
is the tangent space to $M$ at the origin.
Since
\[
(T_0M)^\perp = \{ (d\Sigma_1(0)^*x,\ldots,d\Sigma_m(0)^*x) : x \in \R^n \}
\]
and given the assumption $\BL(d\datum{\Sigma}(0)^*,\datum{p}) < \infty$, Corollary \ref{t:PDEapp} now follows from Proposition \ref{p:FourierBL}.
\end{proof}
\subsection{Abstract harmonic analysis}
In abstract harmonic analysis, the question of determining the best constants in the Young convolution inequality and its cousin, the Hausdorff--Young inequality, has often been discussed, particularly following the work of Babenko \cite{Babenko-Izv-1961}, Beckner \cite{Beckner2, Beckner} and Brascamp--Lieb \cite{BL} which covered the euclidean space $\R^n$.
For example, in the 1970s, Fournier \cite{Fourn} showed that the optimal constant in both inequalities is less than 1 for nonabelian groups with no compact open subgroups, and Klein and Russo \cite{KR} calculated them in terms of the euclidean constants for some nonabelian Lie groups including the Heisenberg groups.
Later, Garc{\'\i}a--Cuerva, Marco and Parcet and then Garc{\'\i}a--Cuerva and Parcet \cite{Garcia-Cuerva-Marco-Parcet-TAMS-2003, Garcia-Cuerva-Parcet-PLMS-2004} considered this and related constants in the context of Banach space structure, but without finding any numerical values.

In recent work, Cowling, Martini, M\"uller and Parcet \cite{CMMP} computed the best constant for the Young convolution inequality and Hausdorff--Young inequality for a more extensive class of nonabelian Lie groups, and established a lower bound for general Lie groups when the functions are restricted to lie in small neighbourhoods of the identity.
More precisely, in the case of the Young convolution inequality, for a connected unimodular Lie group $G$ and triple of exponents $\datum{q} = (q_1,q_2,q_3) \in [1,\infty]^3$ satisfying $\sum_{j=1}^3 1/q_j = 2$, they define the constant $\constantY_\datum{q}(G)$ to be the best value of the constant $C$ in the inequality
\[
\int_{G} \int_{G} f_1(y) f_2(y^{-1}x) f_3(x)  \,dy \,dx
\leq C\prod_{j=1}^3 \|f_j\|_{q_j},
\]
where $f_j $ are nonnegative functions in $L^{q_j}(G)$, for $j=1,2,3$.
The constant $\constantY_\datum{q}(G;U)$ is defined similarly, but the functions $f_1, f_2, f_3$ are constrained to be supported in a relatively compact neighbourhood $U$ of the identity in $G$.
Evidently the constants $\constantY_{\datum{q}}(G;U)$ get smaller as $U$ gets smaller, and so it makes sense to define the constant $\constantY_{\datum{q}}^{\loc}(G)$ to be the infimum of the $\constantY_{\datum{q}}(G;U)$ as $U$ shrinks to the identity.
It is easy to see that
\begin{equation}\label{eq:local-non-local-interpolation}
\constantY_{\datum{q}}^{\loc}(G) \leq \constantY_{\datum{q}}(G) \leq 1;
\end{equation}
the right-hand inequality follows from standard interpolation arguments.

In this definition, the Lebesgue spaces are defined in terms of a left-invariant Haar measure on $G$.
The natural version of Young's inequality for convolution in the context of right-invariant Haar measures on general nonunimodular groups involves powers of the modular function as well; however, the modular function is bounded and bounded away from $0$ on relatively compact neighbourhoods of the identity in $G$, and as the neighbourhood shrinks, the modular function becomes closer to $1$.
As a result, one might also define $\constantY_{\datum{q}}^{\loc}(G)$ with a right-invariant measure, and the value of the constant would be the same as for a left-invariant measure.

In the previous papers on the subject, estimates were obtained for $\constantY_{\datum{q}}(G)$ and $\constantY_{\datum{q}}^{\loc}(G)$ for various types of Lie groups.
In particular, upper bounds were found for both constants on solvable Lie groups, and for similarly defined constants for central functions on compact Lie groups, but for general functions on semisimple Lie groups, say, there were no nontrivial upper bounds for $\constantY_{\datum{q}}(G)$ or $\constantY_{\datum{q}}^{\loc}(G)$ before this work.

By using a contraction argument, it was shown in \cite[Proposition 2.4]{CMMP} that
\begin{equation}\label{eq:general-lower-bound-local}
\constantY_{\datum{q}}^{\loc}(G) \geq (A_{q_1}A_{q_2}A_{q_3})^{\dim(G)},
\end{equation}
where, for any $q \in [1,\infty]$, the constant $A_{q}$ is given by \eqref{e:BecknerC}.

Combined with the inequality \eqref{eq:local-non-local-interpolation} and the estimates for solvable groups, it followed that $\constantY_{\datum{q}}(G)$, for a connected simply connected unimodular solvable group $G$, is exactly the right-hand side of \eqref{eq:general-lower-bound-local}.
In particular, $\constantY_{\datum{q}}(G)$ was known for many interesting groups, such as the Heisenberg groups that appear in several areas of complex analysis and partial differential equations.
However, for many Lie groups, it remained the case that the only upper bound for either constant ($\constantY_{\datum{q}}(G)$ or $\constantY_{\datum{q}}^{\loc}(G)$) was the trivial bound \eqref{eq:local-non-local-interpolation}.
The following corollary of the nonlinear Brascamp--Lieb inequality (Theorem \ref{main}) identifies the local constant $\constantY_{\datum{q}}^{\loc}(G)$ in full generality.
\begin{corollary}\label{LieCor}
Suppose that  $G$ is a connected Lie group.
Then
\begin{equation*}
\constantY_{\datum{q}}^{\loc}(G) = (A_{q_1}A_{q_2}A_{q_3})^{\dim(G)}.
\end{equation*}
\end{corollary}
It is straightforward to verify that Corollary \ref{LieCor} follows from Theorem 1.1, as the differentials of the underlying group multiplication and inversion mappings at the identity correspond to the linear mappings in the classical version of Young's convolution inequality, via the Baker--Campbell--Hausdorff formula.

Young's convolution theorem shows that sharp local nonlinear Brascamp--Lieb inequalities do not in general give rise to sharp global inequalities.
For example, for a compact Lie group $G$, we have shown that $\constantY_{\datum{q}}^{\loc}(G) < 1$, but $\constantY_{\datum{q}}(G) = 1$, so the global constant is larger than the local constant. This is in contrast with the case of homogeneous groups, such as the Heisenberg group, where the local and global constants are easily seen to be the same by scaling. We refer the reader to \cite{CMMP} for much more on this topic, including many references and its connection with the best constant in the Hausdorff--Young inequality. It should be remarked here that there are other examples of global nonlinear Brascamp--Lieb inequalities established with sharp constants -- see for instance Carlen--Lieb--Loss \cite{CLL}, Barthe--Cordero-Erausquin--Maurey \cite{BCM}, Barthe--Cordero-Erausquin--Ledoux--Maurey \cite{BCLM}, and Bramati \cite{Bramati}.

\section{The proof strategy}\label{Sect:3}
As we describe in the introduction, a key element in our proof of Theorem \ref{main} is Theorem \ref{efflie} --- our effective version of Lieb's theorem.
Our mechanism for applying Theorem \ref{efflie} is a rather delicate instance of the method of induction-on-scales.
The purpose of this short section is to explain our general strategy for achieving this, avoiding reference to technical aspects.

The basic applicability of induction-on-scales in the setting of the Brascamp--Lieb inequality is easily seen in a fundamental inequality originating in work of Ball \cite{Ball}.
This inequality captures an important self-similarity property of the \textit{Brascamp--Lieb functional}
\[
\datum{f}\mapsto\BL(\datum{L},\datum{p}; \datum{f})
=\frac{\int_{\R^n}\prod_{j=1}^m f_j^{p_j}(L_j x) \,dx}
        {\prod_{j=1}^m\Bigl(\int_{\R^{n_j}}f_j(x_j) \,dx_j\Bigr)^{p_j}}.
\]
As we shall now see, this property strongly suggests approaching any sort of ``perturbation" of the Brascamp--Lieb inequality by the method of induction-on-scales.
Indeed the method of induction-on-scales has been fruitful on several occasions in this wider context --- see \cite{B,BCT,BHT,BB,Guth2,BBFL}, along with closely-related papers in Fourier restriction theory beginning with \cite{Bo}.

We begin by observing that for two $L^1$-normalised inputs $\datum{f}$ and $\datum{g}$ (in the sense that each $f_j$ and $g_j$ is a probability density),
\begin{equation}\label{Ballineq}
\begin{aligned}
\BL(\datum{L},\datum{p}; \datum{f})\BL(\datum{L},\datum{p}; \datum{g})
&=\int_{\R^n} \Bigl(\prod_{j=1}^m (f_j^{p_j} \circ L_j)\Bigr)*\Bigl(\prod_{k=1}^m (g_k^{p_k}\circ L_k)\Bigr)(x) \,dx\\
&=\int_{\R^n} \Biggl(\int_{\R^n} \prod_{j=1}^m[f_j(L_j y)g_j(L_j(x-y))]^{p_j} \,dy\Biggr) \,dx\\
&=\int_{\R^n} \Biggl(\int_{\R^n} \prod_{j=1}^m[f_j(L_j y)g_j(L_jx - L_jy))]^{p_j} \,dy\Biggr) \,dx\\
&=\int_{\R^n} \Biggl(\int_{\R^n} \prod_{j=1}^m[h_j^x(L_jy)]^{p_j} \,dy\Biggr) \,dx,
\end{aligned}
\end{equation}
where $h_j^x(z) = f_j(z) g_j(L_jx-z)$ for each $j$.
Writing $\datum{h}^x=(h_j^x)_{j=1}^m$ and $\datum{f}*\datum{g}= (f_j*g_j)_{j=1}^m$, we deduce that
\begin{equation*}
\begin{aligned}
\BL(\datum{L},\datum{p}; \datum{f})\BL(\datum{L},\datum{p}; \datum{g})
&\leq \int_{\R^n} \BL(\datum{L},\datum{p}; \datum{h}^x) \prod_{j=1}^m\Bigl(\int_{\R^{n_j}}h_j^x(z_j) \,dz_j\Bigr)^{p_j} \,dx\\
&=\int_{\R^n} \BL(\datum{L},\datum{p}; \datum{h}^x)
    \prod_{j=1}^m (f_j*g_j(L_jx))^{p_j} \,dx\\
&\leq \sup_x \BL(\datum{L},\datum{p}; \datum{h}^x) \BL(\datum{L},\datum{p}; \datum{f}*\datum{g}).
\end{aligned}
\end{equation*}
This we refer to as Ball's inequality.
It should be noted that the $L^1$ normalisation hypotheses on $\datum{f}$ and $\datum{g}$ may be dropped by homogeneity considerations.

Now, if $\datum{g}$ is a $\delta$-\emph{near-extremiser} for \eqref{BL}, in the sense that \eqref{delnear} holds, then we may immediately deduce that
\begin{equation}\label{Ball1}
\BL(\datum{L},\datum{p}; \datum{f})\leq (1+O(\delta))\BL(\datum{L},\datum{p}; \datum{f}*\datum{g})
\end{equation}
and
\begin{equation}\label{Ball2}
\BL(\datum{L},\datum{p}; \datum{f})\leq (1+O(\delta))\sup_x \BL(\datum{L},\datum{p}; \datum{h}^x).
\end{equation}
The inequalities \eqref{Ball1} and \eqref{Ball2} contain a surprising amount of information.
For example, \eqref{Ball1} tells us that the set of extremisers for \eqref{BL} is closed under convolution --- a fact that may be used along with the central limit theorem to deduce the existence of gaussian extremisers (see \cite{BCCT1}), provided that the set of extremisers is nonempty of course.
In the presence of a gaussian extremiser, inequality \eqref{Ball1} may be used again to deduce that the Brascamp--Lieb functional is nondecreasing as the inputs evolve under a suitable heat flow --- see \cite{CLL} and \cite{BCCT1} for more information on this heat-flow perspective on the Brascamp--Lieb inequality.

The closely-related inequality \eqref{Ball2}, while more complicated, will be more important for our purposes.
Informally, if we happen to have a $\delta$-near-extremising input $\datum{g}$ \emph{which resembles a bump function with small support}, then the function
\[
h_j^x(z) = f_j(z) g_j(L_jx-z)
\]
is simply the function $f_j$ localised to a small neighbourhood of the point $L_jx$.
The inequality \eqref{Ball2} then tells us that the Brascamp--Lieb functional is close to increasing ($\delta$ being small) as the general input functions $f_j$ are localised in this way.
Our approach to proving Theorem \ref{main} is motivated by this, and amounts to proving a suitable nonlinear variant of \eqref{Ball2}.
This approach is particularly natural, as input functions $f_j$ with ``shrinking support" will be increasingly unable to detect nonlinear structure in $B_j$, ultimately allowing us to reduce matters to a linear Brascamp--Lieb inequality.
Proving such a nonlinear variant of \eqref{Ball2} presents several difficulties.
Perhaps the most obvious difficulty is in finding a substitute for the explicit use of linearity in \eqref{Ballineq}.
Another difficulty relates to the fact that $(\datum{L}, \datum{p})$ may only have \emph{gaussian} near-extremisers, which needless to say, are not compactly supported. It turns out to be relatively straightforward to overcome this by using suitably truncated gaussians.
The main difficulty we encounter, however, stems from the fact that Lieb's theorem (Theorem \ref{main}) does not provide any quantitative information about the set of gaussian $\delta$-near-extremisers.
In particular, it does not tell us whether there exist $\delta$-near-extremisers that are sufficiently ``localised" to allow the inductive process to run.
What is needed is the availability of gaussian near-extremisers at all scales, whose ``eccentricities" are suitably controlled.
It is here where our effective version of Lieb's theorem (Theorem \ref{efflie}) plays a crucial role.
Once a suitable nonlinear variant of \eqref{Ball2} has been obtained (the forthcoming Proposition \ref{prop:recursive}), then provided the scale $\delta$ is taken from a sufficiently lacunary sequence $(\delta_k)$ converging to zero, the factors of $1+O(\delta)$ may be tolerated, as upon iteration they ultimately lead to a convergent
product $\prod_k(1+O(\delta_k))$, which can be made as close to $1$ as we please.

It should be remarked that for Brascamp--Lieb data where gaussian extremisers exist, such as for the class of so-called \emph{simple data}, there is no need for our effective version of Lieb's theorem. This simplification was helpful in devising the forthcoming induction-on-scales argument -- see the unpublished work \cite{BBBF} which this paper supercedes.

The basic inductive scheme described here should be compared with that in \cite{BHT} and \cite{BB} in the setting of Loomis--Whitney-type data (data satisfying the structural condition \eqref{bas}).
In that case, suitable compactly supported extremisers are available, and an analogous nonlinear version of \eqref{Ball2} is obtained.
The inductive proof that we present here is necessarily different in several important respects.
In particular, and unlike in most induction-on-scales arguments, we completely avoid discrete partitions of unity, and make fundamental use of certain special functions --- in this instance, gaussians.
This makes our approach somewhat closer to that of the heat-flow monotonicity arguments of \cite{CLL}, \cite{BCCT1} and \cite{BCT}.


\section{Exponential optimisation and a proof of Theorem \ref{effexp}}\label{Sect:4}

As preparation for the proof of Theorem \ref{effexp}, we begin in Section \ref{section:pre} by introducing some notation and preliminary observations, after which we state and prove the key lemmas in Section \ref{section:keylemmas}.
Finally, in Section \ref{section:effexpproof} we prove Theorem \ref{effexp} and, in addition, establish that the infimum of weighted sums of exponential functions is locally H\"older continuous with respect to the weights.

\subsection{Preliminaries} \label{section:pre}
Fix a finite index set $J$ and let $(u_j)_{j \in J}$ be a family of pairwise distinct
vectors in $\R^n$.
For each subset $I$ of $J$, we are interested in the functions $f_I$ given by
\begin{align}
	f_{I}(y,d)=\sum_{j\in I}d_j \exp(\langle u_j, y\rangle),
\end{align}
where $d=(d_j)_{j\in J}$, $d_j\geq0$, $y\in\R^n$, and their infimum
\begin{align}
	g_{I}(d)=\inf_{y \in \R^n} f_I(y,d).
\end{align}
In studying the infimum of $f_J(\cdot,d)$, since minimisers do not always exist, it is natural to consider near-minimisers of $f_J(\cdot,d)$ and to what extent we can control their size; if
\[
f_J(y_\delta,d)\leq g_J(d)+\delta \max_{j \in J} d_j
\]
for $\delta>0$, then what can be said about $y_\delta$?
Minimisers are not necessarily unique and can vary on an affine subspace, which means that $y_\delta$ need not be bounded in terms of $\delta$.
Nevertheless, as we will show, we can always find a $\delta$-near-minimiser that is bounded logarithmically in $\delta$.

To study $f_J$, it is very helpful to study functions $f_I$ generated by subsets $I$ of $J$, and we make the trivial observation that if $I_1\subseteq I_2\subseteq J$, then  $f_{I_1}(y,d)\leq f_{I_2}(y,d)$ and therefore
\begin{align}\label{eq:Emonoton}
	g_{I_1}(d)\leq g_{I_2}(d).
\end{align}
For any subset $I$ of $J$, we denote by $K(I)$ the convex hull of the set $\{u_j:j\in I\}$. This is of course a convex polytope in $\mathbb{R}^n$. We denote by $K(I)^{\circ}$ the relative interior of $K(I)$, that is, the interior of $K(I)$ relative to its affine hull, and denote by $\partial K(I)=K(I)\backslash K(I)^\circ$ its relative boundary. The polytope $K(I)$ plays a key role in the proof of Theorem \ref{effexp}.

If the $u_j$ for which $d_j>0$ do not span all of $\R^n$, then $f_I(\cdot ,d)$ is constant along the cosets of the subspace $K(I)^\perp$.
We may, if we wish, restrict attention to the span of $K(I)$.

Clearly, any $j \in J$ for which $d_j=0$ gives no contribution, so it is natural to introduce
\[
J_+=\{j\in J : d_j>0\}
\]
and study the function $f_{J_+}$, since clearly $f_{J_+}(y,d|_{J_+})=f_J(y,d)$.
In fact, we will first study $f_I(\cdot,d)$, where $I\subseteq J$ and $d\in(0,1]^I$.
There is the trivial case in which $I=\emptyset$, that is, $f$ is identically $0$.
It is easy to see that this is the only case where a minimiser exists and $g_I(d)=0$.
Excluding this case, we have the following trichotomy.
\begin{prop}\label{trichot}Let $I\neq\emptyset$.
Then
\begin{itemize}
	\item $0 \in K(I)^{\circ}$ if and only if a minimiser for $f_I(\cdot,d)$ exists and $g_I(d)>0$.
	\item $0 \in \partial K(I)$ if and only if a minimiser for $f_I(\cdot,d)$ does not exist and $g_I(d)>0$.
	\item $0 \notin K(I)$ if and only if a minimiser for $f_I(\cdot,d)$ does not exist and $g_I(d)=0$.
\end{itemize}
\end{prop}
We will give a proof of this proposition after proving Lemmas \ref{lem:goodextremiser} and \ref{lem:nearextrem1} below.
Before we come to these key lemmas, we introduce three positive constants, $C_0$, $c_0$ and $c_1$, which depend on the family $(u_j)_{j \in J}$.
The first of these is simply given by\footnote{For technical reasons, it is more convenient if $C_0\geq 1$, though it is not crucial for the argument.}
\[
C_0=1\vee\max\limits_{j\in J} |u_j|.
\]
The constant $c_0 \in (0,1)$ is chosen such that
\[
\Span(K(I))\cap\bar B(0,c_0)\subseteq  K(I)
\]
for all subsets $I$ of $J$ such that $0\in K(I)^{\circ}$.
We note that since $J$ is a finite set, it is certainly possible to choose such $c_0 \in (0,1)$ that depends only on $(u_j)_{j \in J}$.

The constant $c_1$ is given by the following lemma.

\begin{lemma} \label{l:c1}
There exists a constant $c_1 \in (0,1)$ depending only on $(u_j)_{j \in J}$ with the following property: for each subset $I$ of $J$ such that $0\notin K(I)^{\circ}$,  there exist a subset $I_1$ of $I$ and a unit vector $v \in \R^n$ such that
\begin{align}
\label{e:v1}   \langle v,u_j \rangle = 0\  &\quad\text{if } j\in  I_1, \\
\label{e:v2}   \langle v,u_j \rangle \geq  c_1& \quad\text{if } j\in  I\backslash  I_1,
\end{align}
and either $0\in K(I_1)^{\circ}$ or $0\notin K(I)$ and $ I_1=\emptyset$.
\end{lemma}
\begin{proof}
Suppose that $0\notin K(I)^{\circ}$.
Then either $0$ is contained in a face of the polytope $K(I)$, or is not contained in $K( I)$ at all.
In the first case, there exist a maximal subset $I_1$ of $I$ such that $0\in K(I_1)^\circ$ and $K(I_1)$ is a face of $K(I)$.
Therefore there is a unit vector $v$ that (interpreted as a functional) separates $K( I_1)$ from the rest of $K( I)$, in the sense that
\begin{align*}
\langle v,y \rangle=0 &\qquad \text{if }y\in K( I_1),  \\
\langle v,y \rangle>0 &\qquad \text{if }y\in K( I)\backslash K( I_1),
\end{align*}
and thus \eqref{e:v2} holds as long as we choose $c_1 \in (0,1)$ such that
\begin{equation*}
c_1 \leq \min_I \min_{y \in K(I \setminus I_1)} \langle v,y \rangle,
\end{equation*}
where the minimum in $I$ is taken over all subsets of $J$ such that $0$ is contained in a face of $K(I)$.
Such $c_1 \in (0,1)$ exists since $J$ is a finite set.

In the second case, when $0\notin K( I)$, we proceed similarly and separate $0$ from $K( I)$ by a unit vector $v$, in which case \eqref{e:v2} holds, provided that we choose $c_1 \in (0,1)$ such that
\begin{equation*}
c_1 \leq \min_I \min_{y \in K(I)} \langle v,y \rangle,
\end{equation*}
where the minimum in $I$ is taken over all subsets of $J$ such that $0 \notin K(I)$.
\end{proof}

\subsection{Key lemmas} \label{section:keylemmas}

A major problem that will arise in the study of $\delta$-near-minimisers is when the coefficients $d_j$ are very small compared with the parameter $\delta$.
Initially we shall proceed by considering coefficients $d\in(0,1]^I$ that are bounded away from zero (with the harmless normalisation $d_j \leq 1$), and obtain bounds on certain $\delta$-near-minimisers in terms of the quantity
\begin{align}
	\Delta=\Delta(I,d)=\min\limits_{j\in I} d_j.
\end{align}
We will later decompose our function $f_I$ as $f_{I'}+f_{I''}$, where $I'$ corresponds to ``big" $d_j$ and $I''$ to ``small" $d_j$, in a sense to be made precise later.

When $0\in{K}(I)^{\circ}$, the following lemma says that it is possible to find a minimiser of size comparable to $\log ({|I|}/{\Delta})$.
\begin{lemma}\label{lem:goodextremiser}
Assume that $0\in{K}(I)^{\circ}$, and let
 $d\in(0,1]^I$.
Then there exists $y\in\R^n$ such that
\begin{equation} \label{e:yextremiser}
 f_I(y,d)=g_I(d)
\end{equation}
and
\begin{equation} \label{e:Deltaupper}	
|y|\leq \frac{1}{c_0} \log\frac{|I|}{\Delta}.
\end{equation}
\end{lemma}

\begin{proof}
After restriction to $\Span K(I)$, the function $f_I$ has a unique minimum $y$, as it is strictly convex and $\lim_{s \to +\infty} f_I(sv,d) = + \infty$ for all nonzero vectors $v \in \Span K(I)$.\footnote{We know much more about this case; for instance, the minimum varies real analytically in the $d_i$ and hence $g_I(d)$ is analytic.}

From the definition of $c_0$, it follows that $c_0y/|y|$ is contained in $K(I)$, so we can write
\begin{align}
	c_0 y = |y| \sum_{j\in I} \lambda_j u_j,
\end{align}
where $\lambda_j\in[0,1]$ and $\sum_{j\in I} \lambda_j=1$.
Then
\begin{align}
	c_0|y| = \sum_{j\in I} \lambda_j \langle u_j, y \rangle,
\end{align}
so there exists at least one $j$ (depending on $y$) such that $\langle u_j, y \rangle \geq c_0|y|$.
We conclude that
\begin{align}
	\Delta \exp(c_0 |y|)\leq d_j \exp(\langle u_j, y \rangle) \leq f_I(y,d)=g_I(d)\leq f_I(0,d) \leq |I|,
\end{align}
which immediately gives the bound \eqref{e:Deltaupper}.
\end{proof}

\begin{remark*}
A simple inequality that will be useful below is
\begin{align}\label{repeat}
	f_I(y,d)\leq |J| D \exp(C_0R)
\end{align}
whenever $I \subseteq J$, $|y|\leq R$ and $d \in [0,D]^I$.
\end{remark*}

We now consider the general situation, in which case minimisers may not exist.
However, in this case, the following lemma tells us that we can find a $\delta$-near-minimiser of size comparable to $\log ({1}/({\delta\Delta}))$.
\begin{lemma}\label{lem:nearextrem1}
There exist $C_1>0$ and $\delta_1>0$ with the following property: for all nonempty subsets $I$ of $J$, $\delta\in(0,\delta_1)$ and $d\in(0,1]^I$, there exists $y \in \R^n$ such that
\begin{equation*}
f_{I}(y,d) \leq g_I(d) +\delta
\end{equation*}
and
\begin{equation} \label{e:yupper}
|y|\leq C_1 \log \frac{1}{\delta\Delta}.
\end{equation}
\end{lemma}
\begin{proof}
If $0\in K(I)^{\circ}$, then we may find a minimiser by Lemma \ref{lem:goodextremiser} with the desired bound \eqref{e:yupper}, so it suffices to consider the case in which $0\notin K(I)^{\circ}$.

If $0\notin K(I)^{\circ}$, then we use Lemma \ref{l:c1} to find a subset $I_1$ of $I$ and a unit vector $v \in \R^n$ such that \eqref{e:v1} and \eqref{e:v2} hold.
Let $y_0$ be a minimiser of $f_{ I_1}$ chosen according to Lemma \ref{lem:goodextremiser} (if $I_1=\emptyset$, then $y_0=0$ will trivially fulfil \eqref{e:yextremiser} and \eqref{e:Deltaupper}).
We will modify $y_0$ by moving along $v$, which keeps $ f_{ I_1}$ stable and shrinks the remainder term $ f_{ I}- f_{ I_1}= f_{ I\backslash I_1}$.
Thus we set $y=y_{0}-sv$, where $s$ will be chosen later, and use \eqref{e:v1} and \eqref{e:v2} to estimate
\begin{align*}\label{1107jul1218'}
	f_{ I}(y,d)
	& = 		\sum_{j\in I_1}d_j \exp(\langle u_j, y_0\rangle)+\sum_{j\in I\backslash I_1}d_j \exp(\langle u_j,y_0\rangle -s\langle u_j,v\rangle)  \\
	& \leq  f_{ I_1}(y_0,d)+ f_{ I\backslash I_1}(y_0,d)\exp(-sc_1) \\
	& \leq g_{ I_1}(d)+ |J| \left(\frac{|J|}{\Delta}\right)^{C_0/c_0} \exp(-sc_1). 	
\end{align*}
For the final inequality we used the trivial estimate \eqref{repeat}.
Therefore, from the monotonicity property \eqref{eq:Emonoton} and the choice
\[
s=\frac{1}{c_1}\bigg(\log \frac{1}{\delta} +\log|J|+\frac{C_0}{c_0}\log \frac{|J|}{\Delta} \bigg),
\]
it follows that
\begin{equation} \label{1107jul1218}
f_{ I}(y,d) \leq g_{I_1}(d)+\delta \leq g_I(d)+\delta.
\end{equation}
We note the above choice of $s$ is strictly positive.

To complete the proof, we choose
\begin{equation} \label{e:C1choice}
\delta_1= \frac{1}{|J|}\qquad\text{and }\qquad C_1=\frac{4C_0}{c_0c_1},
\end{equation}
and may easily verify that
\begin{equation*}
	s \leq \frac{3C_0}{c_0c_1} \log \frac{1}{\delta\Delta} \,,
\end{equation*}
whence \eqref{e:yupper} holds.
\end{proof}

Before proving Theorem \ref{effexp}, we establish Proposition \ref{trichot}.

\begin{proof}[Proof of Proposition \ref{trichot}]
First of all, if $0\in K(I)^\circ$, then the existence of a minimiser follows from Lemma \ref{lem:goodextremiser}, and it is easy to see that then $g_I(d)>0$.

Now suppose that  $0 \notin K(I)^\circ$.
By \eqref{1107jul1218},
\begin{align}
	g_I(d) = \inf_{y \in \R^n} f_I(y,d) = g_{I_1}(d),
\end{align}
where either $0\in K(I_1)^{\circ}$ or $I_1=\emptyset$.
Now, from the proof of Lemma \ref{l:c1}, if $0\in \partial K(I)$ then $0\in K(I_1)^{\circ}$ (and $I_1\neq\emptyset$).
This means that the infimum of $g_{I_1}$ is attained and
\[
g_I(d) = g_{I_1}(d) = \min_{y \in \R^n} f_{I_1}(y,d) > 0.
\]
On the other hand, if $0\notin K(I)$, then $I_1=\emptyset$ and hence $g_I(d)=g_{I_1}(d) = 0$.

It remains to show that extremisability implies that $0\in K(I)^\circ$.
To this end, assume that a minimiser for $f_I(\cdot,d)$ exists and $0 \notin K(I)^\circ$.
If the infimum of $f_I(\cdot,d)$ is attained at $y_0$, then
\begin{equation*}
	f_{I\backslash I_1}(y_0,d) = f_{I}(y_0,d)-f_{I_1}(y_0,d)
			\leq g_I(d) - g_{I_1}(d) =0,
\end{equation*}
which means $I=I_1$, and thus contradicts properties of $I_1$.
We deduce that $0\in K(I)^\circ$ is necessary for extremisability, and this concludes the proof of Proposition \ref{trichot}.
\end{proof}

\subsection{Proof of Theorem \ref{effexp} and local H\"older continuity} \label{section:effexpproof}

\begin{proof}[Proof of Theorem \ref{effexp}]
Without loss of generality, we may assume that $d_j\leq 1$ for all $j\in J$.
The proof rests on a partition, depending on $\delta$, of the coefficients $d_j$ into two parts, one for ``big" coefficients and one for ``small" coefficients.
However, it will be important that there is a quantitative ``gap" at the boundary of this partition and therefore our argument is more subtle than simply choosing an appropriate power of $\delta$ as the threshold for the partition.

Let $N_1\geq C_1$ be a parameter that we will specify later, where $C_1$ is the constant arising from Lemma \ref{lem:nearextrem1}.
For nonnegative integers $k$, we define the disjoint subsets $I_k$ of $J$ by
\[
I_k=\{j\in J: \delta^{N_1^{k+1}}\leq d_j< \delta^{N_1^k}\}.
\]
Then, by the pigeonhole principle, there exists $k=k_\delta \in \{0,\ldots,|J|\}$ such that $I_k$ is empty.
For this $k$, we introduce
\[
I'=\{j\in J : d_j\geq\delta^{N_1^{k}}\}
\]
and observe that, since $I_k$ is empty,
\begin{equation*}
J \setminus I' = \{j \in J : d_j\leq\delta^{N_1^{k+1}} \}.
\end{equation*}

Now consider the function $f_{I'}$, where
\[
\Delta=\Delta(d,I')=\min\limits_{I\in I'} d_j \geq\delta^{N_1^{k}}.
\]
According to Lemma \ref{lem:nearextrem1}, there exists $y \in \R^n$ such that
\begin{equation*}
f_{I'}(y,d) \leq g_{I'}(d) + \delta^2
\end{equation*}
and
\begin{equation*}
|y| \leq C_1 \log \frac{1}{\delta^2 \Delta}.
\end{equation*}
Therefore, we have the upper bound
\begin{align}\label{ecc1}
	|y|\leq  C_1(N_1^{|J|}+2) \log \frac{1}{\delta},
\end{align}
and we will show that by choosing $N_1$ and $\delta_0$ appropriately, $y$ is a $\delta$-near-minimiser of $f_{J}$.
To see this, we use the trivial estimate \eqref{repeat} to obtain
\begin{equation*}
	f_{J\backslash I'}(y) \leq |J| \delta^{N_1^{k+1}} \delta^{-C_0C_1(N_1^k+2)}
\end{equation*}
and thus, if we choose $N_1= 4C_0C_1$, then
\begin{equation*}
f_{J\backslash I'}(y)  \leq  |J| \delta^2.
\end{equation*}
This immediately gives
\begin{equation*}
	f_{J}(y,d) \leq g_{ I'}(d)+\delta^2 +|J| \delta^2 \leq g_{J}(d)+\delta,
\end{equation*}
provided that
\begin{align}\label{choiceofd3}
	\delta_0\leq (1+|J|)^{-1}.
\end{align}
Finally, we note that if we choose
\begin{align}\label{choiceofN3}
		N=(4C_0C_1)^{|J|+1},
\end{align}
then $|y| \leq N \log (1/\delta)$ from our choice of $N_1$ and \eqref{ecc1}.
\end{proof}

\begin{remark*}
A review of the proofs of Lemma \ref{lem:nearextrem1} and Theorem \ref{effexp} reveals that an explicit choice of $N$ and $\delta_0$ for which the claim in Theorem \ref{effexp} holds is given by
\begin{align}\label{choiceofN4}
N= \left(\frac{16C_0^2}{c_0c_1}\right)^{|J|+1}		
\qquad\text{ and }\qquad
\delta_0= \frac{1}{|J|+1},
\end{align}
where the constants $C_0$, $c_0$ and $c_1$ are those defined at the start of Section \ref{section:pre}.
\end{remark*}

As an immediate application of Theorem \ref{effexp}, we show that the infimum $g_I$ is a locally H\"older-continuous function.
\begin{theorem}[Local H\"older continuity] \label{hoeldercont1}
There exists a number $\alpha \in (0,1)$, depending only on $(u_j)_{j \in J}$, such that for any $D > 0$ and $I \subseteq J$,
\begin{equation} \label{e:Holdercont}
	|g_I(d)-g_I(d')|\leq (D+|J|+2D|J|\delta_0^{-1})\|d-d'\|_\infty^\alpha
\end{equation}
 for all $d,d'\in [0,D]^I$. Here $\delta_0$ is given by Theorem \ref{effexp}.
\end{theorem}
\begin{proof}
We shall show that the claimed local H\"older continuity holds for $\alpha=(1+C_0N)^{-1}$. We begin by observing that 
\eqref{e:Holdercont} holds easily if $\|d - d'\|_\infty^\alpha \geq \delta_0$ thanks to the elementary bound $g_I\leq D|J|$ on $[0,D]^I$. Thus 
it suffices to show that
\begin{equation} \label{e:HolderETS}
|g_I(d)-g_I(d')|\leq (D + |J|) \|d-d'\|_\infty^\alpha
\end{equation}
whenever $\|d - d'\|_\infty^\alpha < \delta_0$.

To prove \eqref{e:HolderETS}, suppose that  $0 < \|d - d'\|_\infty < \delta_0^{1 + C_0N}$ and define $\delta \in (0,\delta_0)$ by
\[
\delta=\|d-d'\|_\infty^{1/(1+C_0N)}.
\]
By Theorem \ref{effexp}, there exists $y \in \R^n$ such that $|y|\leq N\log(1/\delta)$ and
\[
f_I(y,d)\leq g_I(d)+D\delta.
\]
Using the elementary bound
\begin{align*}
	|f_I(y,d)-f_I(y,d')|\leq |J|\|d-d'\|_\infty e^{C_0|y|}
\end{align*}
and our choice of $\delta$, we thus see that
\begin{align*}
g_I(d')-g_I(d) & \leq f_I(y,d')-f_I(y,d)+D\delta \\
& \leq |J|\|d-d'\|_\infty\delta^{-C_0N}+D\delta	\\
& = (D+|J|)\|d-d'\|_\infty^{1/(1+C_0N)}.
\end{align*}
By symmetry, \eqref{e:HolderETS} follows.
\end{proof}


\section{Proof of Theorem \ref{efflie} and local H\"older continuity of the Brascamp--Lieb constant }\label{Sect:5}

In this section, we shall see how to prove Theorem \ref{efflie}, our effective form of Lieb's theorem, as an application of our abstract near-optimisation result in Theorem \ref{effexp}.
We shall also establish the local H\"older continuity of the Brascamp--Lieb constant as an application of Theorem \ref{hoeldercont1}.
Naturally, both of these deductions are based on having a suitable expression for $\BLg (\datum{L},\datum{p};\datum{A})$ in terms of sums of exponential functions.
The expression we use is contained in Proposition \ref{p:linktoexp} below and is based on ideas from the proof of \cite[Theorem 3.1]{BBCF}.
To state the result, we need to set up some more notation.

First, fix an admissible Brascamp--Lieb datum $(\datum{L},\datum{p})$.
Next, given an $n_j \times n_j$ positive definite matrix $A_j$, we write
\[
 A_j=R_j^*X_jR_j,
\]
where $R_j$ is an $n_j \times n_j$ rotation matrix and $X_j$ is a diagonal matrix with positive diagonal entries.
It is convenient to list the diagonal entries of $X_1, \dots, X_m$ as $e^{y_1},\dots,e^{y_{M}}$, where each $y_j$ is a real number and $M = \sum_{j=1}^m n_j$.

The index set $\Index$ is given by
\[
\Index=\big\{I\subseteq \{1,\ldots,M\}:|I|=n \big\}
\]
and, for each $I \in \Index$, the vector $u_I \in \R^M$ is given by
\[
u_I = \1_I - q.
\]
Here, $\1_I$ denotes the vector $(\1_I(j))_{j=1}^M\in\R^M$, where $\1_I(j) = 1$ if $j \in I$, and $\1_I(j) = 0$ if $j \notin I$.
The vector $q \in \R^M$ takes the form
\[
q = (p_1,\ldots,p_1,p_2,\ldots,p_2,\ldots,p_m,\ldots,p_m),
\]
where the first $n_1$ components coincide with $p_1$, the next $n_2$ components coincide with $p_2$, and so on.
The vector $y \in \R^M$ is simply given by $y = (y_j)_{j=1}^M$, where $y_1,\ldots,y_M$ were introduced above.
Finally, for each $I \in \Index$, we define
\begin{equation} \label{e:dIdefn}
d_I(\datum{L},\datum{p};\datum{R})  = q_I\det( (v_k)_{k \in I})^2
\end{equation}
where $v_k = L_j^*R_j^*e_{(j,\ell)}$, $\datum{R} = (R_j)_{j=1}^m$ and $q_I = \prod_{k \in I} q_k$.
Here $\{ e_{(j,\ell)} : \ell = 1,\ldots,n_j\}$ is the standard basis for $\R^{n_j}$ and we are identifying $k \in \{1,\ldots,M\}$ with the pair $(j,\ell)$ via $k = n_0 + \dots +n_{j-1} + \ell$, where $j \in \{1,\ldots,m\}$, $\ell \in \{1,\ldots,n_j\}$ and $n_0 = 0$.

\begin{proposition} \label{p:linktoexp}
With the notation as above,
\begin{equation} \label{e:linktoexp}
\BLg (\datum{L},\datum{p};\datum{A})^{-2} = \sum_{I \in \Index} d_I(\datum{L},\datum{p};\datum{R}) \exp(\langle  u_I,y \rangle).
\end{equation}
\end{proposition}
\begin{proof}
To see \eqref{e:linktoexp}, we follow the proof of \cite[Theorem 3.1]{BBCF} and write
\[
\sum_{j=1}^m p_j L_j^*A_jL_j = \sum_{k=1}^M q_k e^{y_k} v_k v_k^*.
\]
The expression on the right-hand side coincides with the product of the matrix whose $k$th column is $q_ke^{y_k}v_k$ and the matrix whose $k$th row is $v_k^*$, and thus
\[
\det\bigg(\sum_{j=1}^m p_j L_j^*A_jL_j\bigg) = \sum_{I \in \Index} d_I(\datum{L},\datum{p};\datum{R}) \exp(\langle  \1_I,y \rangle)
\]
follows from the Cauchy--Binet formula.
The identity \eqref{e:linktoexp} is now immediate.
\end{proof}

Before beginning the proof of Theorem \ref{efflie}, we recall that Theorem \ref{effexp} dealt with minimisers and near minimisers for functions of the form
\[
f_\Index(y,d) = \sum_{I \in \Index} d_I \exp(\langle  u_I,y  \rangle)
\]
where $y \in \R^M$ and $d = (d_I)_{I \in \Index}$ is a family of positive coefficients, as well as
\[
g_\Index(d) = \inf_{y \in \R^M} f_\Index(y,d).
\]

\begin{proof}[Proof of Theorem \ref{efflie}]
As a first step to proving our effective version of Lieb's theorem, we observe that it suffices to consider the case where each $L_j$ is a projection; that is, $L_jL_j^*$ coincides with the identity transformation $I_{n_j}$ on $\R^{n_j}$ for each $j=1,\ldots,m$.
To see this, assume that \eqref{effin} holds when the $L_j$ are projections.
Then, given any admissible $\datum{L}$, we define
\[
L'_j= C_j^{-1}L_j
\]
with $C_j = (L_jL_j^*)^{1/2}$.
Then it is easy to check that $\datum{L}'$ is admissible and each $L_j'$ is a projection, and thus \eqref{effin} holds for $\datum{L}'$.
It may be deduced that \eqref{effin} holds for $\datum{L}$ using the elementary fact
\begin{equation*}
\BLg (\datum{L},\datum{p};\textbf{A})^{-2}
	= \BLg (\datum{L'},\datum{p};\textbf{A}')^{-2} \prod_{j=1}^m\det(L_jL_j^*)^{p_j},
\end{equation*}
whenever $\datum{A}$ and $\datum{A}'$ satisfy $A_j=(C_j^{-1})^*A'_jC_j^{-1}$.

We now prove \eqref{effin} under the assumption that $L_jL_j^*= I_{n_j}$.
From Proposition \ref{p:linktoexp}
\[
\BLg (\datum{L},\datum{p};\textbf{A})^{-2} = f_\Index(y,d(\datum{L},\datum{p};\datum{R})),
\]
where $d(\datum{L},\datum{p};\datum{R}) = (d_I(\datum{L},\datum{p};\datum{R}))_{I \in \Index}$ and therefore, from Lieb's theorem (Theorem \ref{LIEB}),
\begin{align} \label{eqn:BLisoptim2}
	\BL(\datum{L},\datum{p})^{-2}
	= \inf_\datum{R} \inf_{y\in\R^M} f_\Index(y,d(\datum{L},\datum{p};\datum{R}))
	= \inf_{\datum{R}} g_\Index(d(\datum{L},\datum{p};\datum{R})).
\end{align}
Since $d(\datum{L},\datum{p};\datum{R})$ is a continuous function of $\datum{R}$ and $g_\Index(d)$ is a continuous function of $d$ by Theorem \ref{hoeldercont1} (this is tantamount to the continuity of the Brascamp--Lieb constant observed in \cite{BBCF}), there exists some $\datum{R}$ (depending on $\datum{L}$ and $\datum{p}$) such that
\begin{align} 
	\BL(\datum{L},\datum{p})^{-2} = g_\Index(d(\datum{L},\datum{p};\datum{R})).
\end{align}
For the remainder of the proof, $\datum{R} = (R_j)_{j=1}^m$ denotes this minimising collection of rotations.

To complete the proof of Theorem \ref{efflie}, we apply Theorem \ref{effexp}.
Thus, there exist $N \in \N$ and $\delta_0 > 0$, depending only on $\datum{p}$, such that, for all $\delta \in (0,\delta_0)$,
\[
\inf_{|y| \leq N \log (1/\delta)} f_\Index(y,d(\datum{L},\datum{p};\datum{R})) \leq g_\Index(d(\datum{L},\datum{p};\datum{R})) + \delta \max_I d_I(\datum{L},\datum{p};\datum{R}).
\]
From the assumption $L_jL_j^*= I_{n_j}$, it follows that $d_I(\datum{L},\datum{p};\datum{R}) \leq 1$ and therefore there exists some $y = y(\delta,\datum{L},\datum{p}) \in \R^M$ satisfying
\begin{equation} \label{e:ysize}
|y| \leq N \log \frac{1}{\delta}
\end{equation}
and
\[
f_\Index(y,d(\datum{L},\datum{p};\datum{R})) \leq g_\Index(d(\datum{L},\datum{p};\datum{R})) + \delta.
\]
Hence
\begin{align}
	\BLg(\datum{L},\datum{p};\textbf{A})^{-2}
	= f_\Index(y,d(\datum{L},\datum{p};\datum{R})) \leq g_\Index(d(\datum{L},\datum{p};\datum{R})) + \delta
	= \BL(\datum{L},\datum{p})^{-2} + \delta.
\end{align}
Here, $A_j=R_j^*X_jR_j$, where $X_1$ is the diagonal matrix with entries $e^{y_1},\ldots,e^{y_{n_1}}$, $X_2$ is the diagonal matrix with entries $e^{y_{n_1+1}},\ldots,e^{y_{n_1+n_2}}$, and so forth.
From \eqref{e:ysize}, it is easily checked that $\|A_j\| \leq (1/\delta)^N$ and $\|A_j^{-1}\| \leq (1/\delta)^N$.
Hence \eqref{effin} holds when $L_jL_j^*= I_{n_j}$, and this completes our proof of Theorem \ref{efflie}.
\end{proof}

To end this section, we improve the result in \cite{BBCF}, which established the continuity of the Brascamp--Lieb constant, to local H\"older continuity.
Given that differentiability of the Brascamp--Lieb constant may fail (see \cite{BBCF}), it seems to be a challenging problem to improve upon the following result.
\begin{theorem}\label{blho}
For each $\datum{p}$, the Brascamp--Lieb constant $\BL(\cdot,\datum{p})$ is locally H\"older continuous. More precisely,
there exists a number $\alpha \in (0,1)$ and a constant $C_0>0$, depending only on $(n_j)_{j=1}^m$ and $(p_j)_{j=1}^m$, such that the following holds.
For any $C > 0$,
\begin{align}\label{rick}
	|\BL(\datum{L},\datum{p})^{-2}-\BL(\datum{L'},\datum{p})^{-2}|
	\leq C_0 C^{n+\alpha(n-1)}\|\mathbf{L}-\mathbf{L}'\|^\alpha
\end{align}
for all data $\mathbf{L},\mathbf{L}'$ satisfying the bound $\|\mathbf{L}\|,\|\mathbf{L}'\|\leq C$.
In particular, if additionally the Brascamp-Lieb constants $\BL(\datum{L},\datum{p})$, $\BL(\datum{L'},\datum{p})$ are bounded above by a constant $C_1$, we have
\begin{align}\label{morty}
	 |\BL(\datum{L},\datum{p})-\BL(\datum{L'},\datum{p})|
		\leq C_0 C_1^3 C^{n+\alpha(n-1)}\|\mathbf{L}-\mathbf{L}'\|^\alpha.
\end{align}
\end{theorem}
We remark that Theorem \ref{blho} establishes global H\"older continuity of the function  $\BL(\cdot,\datum{p})^{-2}$ when restricted to projection data $\datum{L}$, for example.
\begin{proof}
First, we note that it is straightforward to verify that \eqref{rick} implies \eqref{morty}.

Next, we observe that each $d_I(\datum{L},\datum{p};\datum{R})$, defined in \eqref{e:dIdefn}, is locally Lipschitz in $\datum{L}$ uniformly in $\datum{R}$ with Lipschitz-constant $c_n C^{n-1}$, where $c_n$ depends only on the dimension $n$.
Also, $g_\Index$ is locally H\"older continuous by Theorem \ref{hoeldercont1} and thus it obviously follows that $g_\Index(d(\cdot,\datum{p};\datum{R}))$ is locally H\"older continuous uniformly in $\datum{R}$ with constant $C_0C^{n+\alpha(n-1)}$.
By elementary considerations, it follows that $\inf_{\datum{R}} g_\Index(d(\cdot,\datum{p};\datum{R}))$ is locally H\"older continuous too, with the same bound.
However, by \eqref{eqn:BLisoptim2}, it follows that $\BL(\cdot,\datum{p})^{-2}=\inf_{\datum{R}} g_\Index(d(\cdot,\datum{p};\datum{R}))$.
\end{proof}

The next simple example may cast a little light on our previous results, and in particular show that sharper results, such as Lipschitz continuity, cannot hold.

\begin{example}
Define $f: \R^4 \to \R^+$ and $g(a,b)$ by
\[
f(x,y, a, b) = a e^x + e^{-x} + e^{x-y} + b e^{y}
\quad\text{and}\quad
g(a,b) = \inf\{ f(x,y,a,b) : (x, y) \in \R^2 \},
\]
where $a, b \in [0,1]$.
We will find when the infimum is attained and how it depends on $a$ and $b$, as well as examine approximate minimisers where we impose the constraint $|x| \leq R$ and $|y| \leq R$.

There are various cases to consider.

\subsubsection*{Case 1: $a = b = 0$}
Evidently,
\[
f(x,y,0,0) =  e^{-x} + e^{x-y} ,
\]
and $g(0,0) = 0$.
For $\delta \in (0,1]$, we take $(x_\delta, y_\delta) = (\log(2/\delta), 2 \log(2/\delta))$; then
\[
f(x_\delta, y_\delta,0,0) = \delta \leq g(0,0) + \delta
\]
and $|x_\delta| \leq 2 \log(2/\delta)$ and $|y_\delta| \leq 2 \log(2/\delta)$.

\subsubsection*{Case 2: $a \neq 0$ and $b = 0$}
In this case,
\[
f(x,y,a,0) =  ae^{x} + e^{-x} + e^{x-y} ,
\]
and evidently
\[
g(a,0)  = \min \{a e^x + e^{-x} : x \in \R \} = 2 a^{1/2}.
\]
Observe that
\[
g(a,0) - g(0,0) = 2 a^{1/2},
\]
so $g(\cdot ,0)$ is H\"older continuous of exponent $1/2$, but no more.
Given $\delta \in (0,1)$, we take $x_\delta =  - \frac{1}{2} \log(a + \delta^2/4)$ and $y_\delta = - 2\log(\delta/2)$.
Then
\[
f(x_\delta, y_\delta)
= \frac{2a + \delta^2/2 }{(a+\delta^2/4)^{1/2}} \leq g(a,0) + \delta
\]
and $|x_\delta| \leq  \log(2/\delta)$ and $|y_\delta| \leq 2\log(2/\delta)$.

\subsubsection*{Case 3: $a = 0$ and $b \neq 0$}
In this case,
\[
f(x,y) =  e^{-x} + e^{x-y} + b e^{y}.
\]
The function $f(x,y)$ attains its minimum at $(\frac{1}{3}\log(1/b) , \frac{2}{3} \log(1/b))$, and
\[
g(0,b) = 3 b^{1/3}.
\]
Again $g(0, \cdot)$ is H\"older continuous of exponent $1/3$, but no more.

Given $\delta \in (0,1)$, define
\[
(x_\delta, y_\delta) =
\begin{cases}
(\frac{1}{3}\log(1/b) , \frac{2}{3} \log(1/b)) & \text{if $b > (\delta/3)^3$} \\
(\log(3/\delta), 2 \log(3/\delta))             & \text{if  $b \leq (\delta/3)^3$} .
\end{cases}
\]
If $b > (\delta/3)^3$, then $f(- \frac{1}{3}\log(b) , - \frac{2}{3} \log(b), 0, b) - g(0,b) = 0$, and otherwise
\[
f(x_\delta, y_\delta, 0, b) - g(0,b)
= 2\delta/3  + b (3 / \delta)^2 - 3 b^{1/3}
\leq  
  \delta - 3 b^{1/3}
< \delta.
\]
Moreover, in both cases, $|x_\delta| \leq \log(3/\delta)$ and $|y_\delta| \leq 2\log(3/\delta)$.
%
%
%

\subsubsection*{Case 4: $a > 0$ and $b > 0$}
This case is more complex and we cannot find formulae that are as precise as those in the previous cases.
Nevertheless, 
we can show that
$f(x,y,a,b)$ achieves its minimum at $(x_*, y_*)$, where
\begin{align*}
e^{{y_*}} &= b^{-2/3} \tau(ab^{-2/3})^{-1} \\
e^{{x_*}} &= b^{-1/3} \tau(ab^{-2/3})^{-2};
\end{align*}
here $\tau(c)$ is the unique positive solution of the equation $t^4 - t = c$.
With somewhat more effort, it is possible to show that if $\delta \in (0,1/e)$, then there exists $(x_\delta, y _\delta)$ such that
\[
f(x_\delta, y_\delta, a, b) - g(a,b) \leq \delta
\]
and $|x_\delta| \leq 4 \log(1/\delta)$ and $|y_\delta| \leq 4 \log(1/\delta)$.

\end{example}


\section{The nonlinear Brascamp--Lieb inequality}\label{Sect:6}
Before we embark on the proof of Theorem \ref{main}, we offer further clarifying remarks.

First, it is not possible to take $\epsilon<0$ in the statement of Theorem \ref{main} for \emph{any} datum $(\datum{L},\datum{p})$, as may be seen by a slight modification of the argument in Lemma 6 from \cite{BBG}; see also \cite{CMMP}.
This argument is related to the so-called ``transplantation" method --- see for example \cite{KST}.
As may be expected, it is not in general possible to take $\epsilon=0$ either, as $\BL(\differential\datum{B}(x),\datum{p})$ varies with $x$ in all but very special situations.

Theorem \ref{main} is stated for $C^2$ maps merely for simplicity.
A careful inspection of our proof will reveal that a $C^{1+\theta}$ condition suffices for any $\theta>0$.

As our proof also reveals, and as may be expected, the neighbourhood $U$ appearing in the statement of Theorem \ref{main} depends only on the underlying linear datum $(\mathbf{L}, \mathbf{p})$ and regularity bounds on the maps $B_j$ (be they $C^2$ or $C^{1+\theta}$ bounds).
From the point of view of potential applications, it is also appropriate to point out that, for fixed $\mathbf{p}$, the neighbourhood $U$ may be taken to be locally uniform in $\mathbf{L}$ --- that is, given $\mathbf{L}_0$ and $A>0$,
there exists a neighbourhood $U$ of $x_0$ such that \eqref{nlbli} holds for all $B_j$ such that $\|B_j\|_{C^2(U)}\leq A$ (or $\|B_j\|_{C^{1+\theta}}\leq A$) and $dB_j(x_0)=L_j$, provided $\mathbf{L}$ is sufficiently close to $\mathbf{L}_0$. Again, this follows from a close inspection of our arguments.

\subsection{Setting up the induction}\label{sect:settingup}
By translation invariance, we may assume that $x_0=0$ in the statement of Theorem \ref{main}.
As discussed in Section \ref{Sect:3}, our proof of Theorem \ref{main} will proceed by induction on the size of the supports of the input functions $\datum{f}$.
In order for there to be a base case for this induction it will be convenient to suppose that $\datum{f}$ satisfies a certain auxiliary regularity condition.
The passage to general integrable $\datum{f}$ will consist of an elementary limiting argument.

\begin{definition}\label{localconst} Suppose that  $\kappa>1,\mu>0$ and $\Omega$ is a measurable subset of $\R^d$.
A nonnegative function $f$ is \textit{$\kappa$-constant at scale $\mu$ on $\Omega$} if $f(x)\leq\kappa f(y)$ whenever $x\in\Omega$ and $y\in\R^d$ are such that $|x-y|\leq\mu$.
We denote by $L^1(\Omega;\mu,\kappa)$ the subset of $L^1(\R^d)$ with this property.
\end{definition}
The reader will observe that the definition above lacks symmetry in $x$ and $y$: we stipulate that $x\in\Omega$, while $y$ may extend a distance $\mu$ from $\Omega$.
The reason for this technicality will become clear shortly.
At this stage it is worth noting that for any fixed $\kappa>1$, a nonnegative function $f\in L^1(\Omega)$ may be approximated almost everywhere by functions in $L^1(\Omega;\mu,\kappa)$, provided that $\mu$ is taken sufficiently small.
One way to see this is to observe that the $d$-dimensional Poisson kernel $P_t(x)$ is $\kappa$-constant at scale $\mu$ on the whole of $\R^{d}$ provided that $\mu$ is small enough depending on $t>0$ and $\kappa>1$.
By linearity, the convolution $f*P_t$ is easily seen to inherit this regularity property for any nonnegative $f\in L^1(\R^{d})$, and so the almost everywhere approximation claimed follows from the Lebesgue differentiation theorem.

We will allow $\kappa$ to vary in a controlled manner through the induction.
This flexibility in $\kappa$ is required when considering products of such ``locally constant" functions.
In particular, we will need to appeal to the elementary fact that if $f\in L^1(\Omega;\mu,\kappa)$ and $g\in L^1(\Omega;\mu,\lambda)$, then $fg\in L^1(\Omega;\mu,\kappa\lambda)$.
Considerations of this type naturally arise when introducing partitions of unity, as we shall to pass between scales.

We now set up the induction.
For each $\delta \in (0, 1)$ and $y\in\R^n$, let
\begin{equation}
U_\delta(y) = \{x \in \R^n : |x-y|\leq\cutscl \}.
\end{equation}

\begin{definition}\label{defC}
For $u\in \R^n$, $\delta>0$, $\mu> 0$ and $\kappa>1$, let $\constantC(u,\delta,\mu,\kappa)$ denote the best constant $C$ in the inequality
\[
\int_{U_\delta(u)}\prod_{j=1}^m f_j^{p_j} ( B_j(y))  \,dy
\leq C\prod_{j=1}^m\left(\int_{\R^{n_j}} f_j(x_j) \,dx_j\right)^{p_j}
\]
over all inputs $f_j\in L^1(B_j(U_{2\delta}(u));\mu,\kappa)$.
\end{definition}
We think of $\constantC(u,\delta,\mu,\kappa)$ as a regularised and localised nonlinear Brascamp--Lieb constant; here we are of course suppressing the dependence on $\datum{B}$ and $\datum{p}$ in our notation.

The requirement that $f_j\in L^1(B_j(U_{2\delta}(u));\mu,\kappa)$ may seem unusual as the integral on the left-hand side does not see the part of $f_j$ supported outside of $B_j(U_\delta(u))$, whereas the right-hand side does.
This merely technical feature will be important for closing the induction.

If $\delta$ is below a certain threshold in terms of $\mu$, then each $f_j\in L^1(B_j(U_{2\delta}(u));\mu,\kappa)$ will (effectively) cease to distinguish between $B_j$ and its best affine approximation $L^u_j$ at $u$, given by
\[
L^u_j:=B_j(u)+dB_j(u)(\cdot-u).
\]
This allows us to reduce the estimation of $\constantC(u,\delta,\mu,\kappa)$ to an application of the \emph{linear} Brascamp--Lieb inequality, and provides us with an effective ``base case" for the inductive argument.
This base case is contained in the first of our two propositions below, and the second contains the recursive inequality for the main step of the induction.

The statements of the two subsequent propositions and their proofs contain certain indices which we now introduce.
The indices are $\alpha, \beta, \gamma, \tau$ and $\sigma$, which should be regarded as fixed, are chosen in this order as follows.
First, we fix $\alpha$ satisfying $1 < \alpha < 2$, and then $\beta$ such that  $0 < \beta < 2-\alpha$.
Next, $\gamma$ is chosen to satisfy
\begin{equation}\label{satisfied?!}
0 < \gamma < \min\left\{\frac{2-\alpha-\beta}{(n+2)\alpha},\frac{\beta}{2\alpha}\right\},
\end{equation}
and then $\tau$ such that
\begin{equation}\label{taucond}
0 < \tau < \frac{\gamma}{N},
\end{equation}
 where $N \in \N$ is given by Corollary \ref{effliecor}.
Finally, $\sigma$ is chosen to satisfy $0 < \sigma < \min\{\beta - 2\alpha\gamma,\alpha \tau\}$.
The reader is encouraged not to dwell on these restrictions at this stage.
We simply note here that the crucial recursive inequality stated below in Proposition \ref{prop:recursive} compares the function $\constantC$ at scales $\delta$ and $\delta^\alpha$; the roles of the other exponents are of a more technical nature, so we postpone further remarks of this type.

\begin{notation}
In this section, we refer to a positive real number $c$ as a \textit{constant} if it depends on at most $(\datum{B}$, $\datum{p})$, and the parameters $\alpha$, $\beta$, $\gamma$, $\tau$ and $\sigma$. For $\epsilon > 0$, we say that $c>0$ is a \textit{constant depending on $\epsilon$} if it is a constant that may depend additionally on $\epsilon$.  Finally, we write $A \lesssim B$ and $B \gtrsim A$ to mean that $A \leq cB$, where $c>0$ is  a constant, and $A \sim B$ means that both $A \lesssim B$ and $A \gtrsim B$. The reader may wish to observe that in all instances where such constants occur, they may in fact be chosen to depend locally uniformly on the nonlinear maps $\datum{B}$, in line with our remark at the beginning of this section.
\end{notation}

Before stating the base case for the induction, we note that by the $C^2$ regularity\footnote{This is the main use of the $C^2$ regularity hypothesis on $\datum{B}$, and where the reader may wish to weaken it to $C^{1,\theta}$ for some $\theta>0$. This would require that the condition $\alpha+\gap<2$ be tightened to $\alpha+\gap<1+\theta$, and the relation $\loss<\gap <2-\alpha$ be replaced with $\loss<\gap <1+\theta-\alpha$.
Finally, $\gamma$ must also satisfy $\gamma< (1+\theta-\alpha-\gap)/(\alpha(n+2))$.} of the $B_j$,
\begin{equation} \label{taylor}
|B_j(x) - L^u_jx| \lesssim |x-u|^2
\end{equation}
for all $x, u \in U_{\nu}(0)$ (with an appropriately small constant $\nu$).
We make frequent use of this fact.

\begin{prop}[Base case] \label{prop:basecase}
There exists a constant $\nu \simeqone$ with the following property:
if $u \in U_{\nu}(0)$, $\kappa > 1$, and if $\delta \in (0,\nu)$ and $\mu > 0$ satisfy $\delta^{\alpha+\gap}\leq\mu$, then
\[
\constantC(u,\delta,\mu,\kappa) \leq \kappa^\sump \BL(\differential\datum{B}(u), \datum{p}),
\]
where $\sump = \sum_{j=1}^m p_j$.
\end{prop}
\begin{proof} Let $\nu\simeqone$ be a small constant, to be chosen below.
We fix $u \in U_{\nu}(0)$, and assume that $\delta \in (0,\nu)$ and $\mu \geq \delta^{\alpha+\gap}$.
We also fix $f_j\in L^1(B_j(U_{2\delta}(u));\mu,\kappa)$.

By choosing  the constant $\nu\simeqone$ sufficiently small and applying \eqref{taylor},
\[
|B_j(y)-L^u_jy| \lesssim \cutscl^2
\]
for all $y\in U_\delta(u)$.
Together with our choices that $\alpha+\gap<2$ and $\delta^{\alpha+\gap} \leq\mu$, this implies that
$
|B_j(y)-L^u_jy| \leq \mu,
$
for a possibly smaller choice of the constant $\nu$.
Now $f_j\in L^1(B_j(U_{2\delta}(u));\mu,\kappa)$, and so $f_j(B_j(y))\leq \kappa f_j(L^u_jy)$ for all $y\in U_\delta(u)$, and so
\begin{align*}
\int_{U_{\delta}(u)}\prod_{j=1}^m f_j^{p_j}(B_j(y))  \,dy & \leq \kappa^\sump \int_{U_{\delta}(u)}\prod_{j=1}^m f_j^{p_j}(L^u_jy)  \,dy \\
& \leq \kappa^\sump \BL(\differential \textbf{B}(u), \datum{p}) \prod_{j=1}^m\left(\int_{\R^{n_j}}f_j(x_j) \,dx_j\right)^{p_j},
\end{align*}
as required.
Here we have used the translation-invariance of the linear Brascamp--Lieb inequality in the $f_j$ to remove the translations in the affine approximation above.
\end{proof}

Our objective is to prove a suitable recursive inequality involving the function $\constantC$, which upon iteration will establish Theorem \ref{main}.
This inequality, which we now state, may be viewed as a certain nonlinear version of \eqref{Ball2}.
\begin{prop}[Recursive inequality] \label{prop:recursive}
There exists  a constant $\nu \simeqone$ with the following property:
if $u \in U_{\nu}(0)$, $\kappa > 1$, and if $\delta \in (0,\nu)$ and $\mu > 0$ satisfy $\delta^{\alpha+\gap}>\mu$, then
\begin{equation}\label{inductionstep}
\constantC(u,\delta,\mu,\kappa) \leq (1 + \delta^{\loss}) \max_{x\in U_{2\delta}(u)} \constantC(x,\delta^\alpha,\mu,\kappa\exp(\delta^\loss)).
\end{equation}
\end{prop}
The reader may wish to view \eqref{inductionstep} as a certain ``near-monotonicity property" of the function $\constantC$ in the parameter $\delta$, since the multiplicative factors $(1 + \delta^{\loss})$ and $\exp(\delta^\loss)$ are close to $1$ for small $\delta$.
It is interesting to compare \eqref{inductionstep} with the related near-monotonicity results in \cite{BCT}, \cite{BHT}, \cite{BB} and \cite{TaoBil}.

We conclude this section by showing how Theorem \ref{main} follows from Propositions \ref{prop:recursive} and \ref{prop:basecase}.

\begin{proof}[Proof of Theorem \ref{main}]
Fix $\epsilon>0$ and set $\kappa=1+\epsilon$.
By an elementary density argument, which we now sketch, it will be enough to show that there exist a constant  $\delta_0 \simeqepsilonone$ depending on $\epsilon$ and a constant $C \simeqone$ such that
\begin{equation}\label{nuf}
\constantC(0,\delta_0,\mu, 1+\epsilon)\leq (1+C\epsilon)\BL(\differential \datum{B}(0), \datum{p})
\end{equation}
uniformly in $\mu>0$.
As we observed above, for each $t>0$, and $f_j\in L^1(2B_j(U_{\delta_0}(0)))$, the Poisson extension satisfies $f_j*P^{(n_j)}_t\in L^1(\R^{n_j}; \mu, \kappa)$, and so after restriction $f_j*P^{(n_j)}_t \in L^1(2B_j(U_{\delta_0}(0));\mu, \kappa)$, provided that $\mu$ is chosen to be sufficiently small.
Hence \eqref{nuf}, combined with the $L^1$ normalisation of the $n_j$-dimensional Poisson kernel $P_t^{(n_j)}$, imply that
\[
\int_{U_{\delta_0}(0)}\prod_{j=1}^m \Bigl(f_j*P^{(n_j)}_t (B_j(x))\Bigr)^{p_j} \,dx
\leq (1+C\epsilon)\BL(\differential \datum{B}(0), \datum{p})
        \prod_{j=1}^m\left(\int_{\R^{n_j}}f_j(x_j) \,dx_j\right)^{p_j}
\]
uniformly in $t>0$.
Since the nonnegative functions $f_j*P^{(n_j)}_t$ converges to $f_j$ almost everywhere as $t\to 0$, Theorem \ref{main} follows by Fatou's lemma.
We now turn to the proof of \eqref{nuf}.

We choose a sufficiently small constant $\delta_0 \simeqepsilonone$, depending on $\epsilon$, to satisfy a number of constraints.
The first of these is $\delta_0\leq \nu/3$, where $\nu$ is sufficiently small that  both Propositions \ref{prop:basecase} and \ref{prop:recursive} hold (that this is possible follows from the fact that both propositions implicitly require that $\nu$ is sufficiently small).

Given $\mu>0$, either the base case provided by Proposition \ref{prop:basecase} holds directly, that is, $\delta_0^{\alpha+\gap}\leq \mu$, or $\mu$ satisfies the threshold condition that $\delta_0^{\alpha+\gap}>\mu$.
In the former case there is nothing more to prove.
In the latter, we shall use the recursive inequality \eqref{inductionstep} of Proposition \ref{prop:recursive} to connect $\constantC(0,\delta_0,\mu, 1+\epsilon)$ with the base case.

We define $\delta_k=\delta_0^{\alpha^{k}}$, where $k \in \N$, and apply Proposition \ref{prop:recursive} to see that
\begin{equation}\label{toiterate}
\constantC(u,\delta_k,\mu,\kappa) \leq (1 + \delta_k^{\loss}) \max_{x\in U_{2\delta_k}(u)} \;\constantC(x,\delta_{k+1},\mu,\kappa\exp(\delta_k^{\loss}))
\end{equation}
for all $u$ in the algebraic sum
$
\widetilde{U}_k:=U_{2\delta_0}(0)+U_{2\delta_1}(0)+\dots+U_{2\delta_{k-1}}(0)
$
and
$k$ such that $\delta_k^{\alpha+\gap}>\mu$.
Here we see the reason for the constraint $\delta_0\leq \nu/3$ (rather than the seemingly more natural $\delta_0\leq\nu$): \eqref{toiterate} requires that $\widetilde{U}_k\subseteq U_\nu(0)$, which is indeed the case for sufficiently small $\delta_0$ independent of $k$.

Choosing $k_*$ to be the smallest natural number $k$ for which $\delta_k^{\alpha+\gap}\leq\mu$, we may iterate \eqref{toiterate} $k_*$ times to obtain
\[
\constantC(0,\delta_0,\mu, 1+\epsilon) \leq \prod_{k=0}^{k_*-1}(1 + \delta_k^{\loss}) \max_{u\in \widetilde{U}_{k_*}} \constantC\Bigl(u,\delta_{k_*},\mu,(1+\epsilon)\prod_{k=0}^{k_*-1}\exp(\delta_k^{\loss})\Bigr).
\]
Since $\delta_{k_*}^{\alpha+\gap}\leq\mu$ we may now deduce from Proposition \ref{prop:basecase} that
\begin{equation} \label{almostthere}
\constantC(0,\delta_0,\mu, 1+\epsilon) \leq (1+\epsilon)^\sump \max_{u\in \widetilde{U}_{k_*}} \BL(\differential \datum{B}(u), \datum{p}) \prod_{k=0}^{k_*-1}\left(1 + \delta_k^{\loss}\right)\exp(\sump\delta_k^{\loss}) .
\end{equation}
At this point, it is clear that we need the Brascamp--Lieb constant to behave sufficiently nicely with respect to the linear mappings.
Fortunately, the continuity result for general data proven recently in \cite{BBCF} and \cite{GGOW} is sufficient here, and ensures that
\[
\max_{u\in \widetilde{U}_{k_*}} \BL(\differential \datum{B}(u), \datum{p})\leq (1+\epsilon)\BL(\differential \datum{B}(0), \datum{p})
\]
for an appropriate choice of a constant  $\delta_0 \simeqepsilonone$ (depending on $\epsilon$).

For the product term in \eqref{almostthere}, note that
\begin{equation*}
\log\left( \prod_{k=0}^{k_*-1}(1+\delta_k^\loss)\exp(\sump\delta_k^{\loss})\right) \leq (1+\sump)\sum_{k=0}^{\infty} \delta_k^{\loss} = (1+\sump)\sum_{k=0}^{\infty} \delta_0^{\alpha^k\loss}\lesssim \delta_0^{\loss}
\end{equation*}
and so $\prod_{k=0}^{k_*-1}(1+\delta_k^\loss)\exp(\sump\delta_k^{\loss}) \leq 1 + \epsilon$ , for an appropriate $\epsilon$-dependent choice of constant $\delta_0 \simeqepsilonone$.
Putting together the above, we have shown that
\[
\constantC(0,\delta_0,\mu, 1+\epsilon)
\leq (1+\epsilon)^{\sump + 3} \BL(\differential \datum{B}(0), \datum{p}),
\]
which establishes \eqref{nuf}, and hence Theorem \ref{main}.
\end{proof}

It remains to prove Proposition \ref{prop:recursive}.

\subsection{Gaussian preliminaries}\label{sect:gaussians}
Before embarking on our proof of Proposition \ref{prop:recursive}, we introduce a family of $m$-tuples of gaussians that will be pivotal in our argument.

As described in Section \ref{Sect:3}, our proof of Proposition \ref{prop:recursive} may be interpreted as a certain nonlinear version of the proof of the inequality \eqref{Ball2} for the linear Brascamp--Lieb constant.
As the nonlinear Brascamp--Lieb constant $\constantC(u,\delta,\mu, k)$ from Definition \ref{defC} involves a localisation to the $\delta$-ball $U_\cutscl(u)$, the gaussians that we introduce here will be near-extremisers for the underlying linear datum $(\differential\datum{B}(u),\datum{p})$, and furthermore, will be such that they remain near-extremisers when truncated to fit $U_\cutscl(u)$.
This requires that we choose the gaussian near-extremisers to be ``nearly isotropic" with variance ``close to" $\delta^{1+\gamma}$ for some suitable $\gamma>0$ (as will be become apparent, the various constraints of our inductive scheme demand that we choose $\gamma$ satisfying \eqref{satisfied?!}).
Corollary \ref{effliecor} will guarantee the existence of such gaussian near-extremisers.
Of course, in order for the resulting gaussian near-extremisers to be ``sufficiently close to isotropic", they should correspond to positive definite matrices that are, modulo scalings,
 ``sufficiently close to the identity".
This requires us to use $\delta^\tau$-near extremisers, where the small exponent $\tau$ satisfies \eqref{taucond}.
This subsection clarifies these heuristic ideas, and establishes several key lemmas that will be used in the forthcoming proof of Proposition \ref{prop:recursive}.

Note the dual role of $\delta$ in this set-up: it determines both the accuracy and the scale of the near-extremiser.

Potentially  many gaussians are in play here.
Not only do we typically have a different gaussian at every point $u$, but we may also have a different gaussian at every scale, at every point.
We note that this second layer of complexity in our family of gaussians may sometimes be removed in cases where extremisers exist.
One such special case is when the family of linear maps $\differential\datum{B}(0)$ satisfies the Loomis--Whitney-type kernel condition \eqref{bas}.
Another special case is when $(\differential\datum{B}(0),\datum{p})$ is \emph{simple} --- that is, when \eqref{char} holds with strict inequality for all nontrivial proper subspaces $V$.
In each of these situations our proof of Theorem \ref{main} admits some simplification.
For example, in the case of simple data, the role of Theorem \ref{efflie} may be replaced with certain regularity properties of gaussian extremisers due to Valdimarsson \cite{ValdBestBest}. We refer to \cite{BBBF} for the details of this.

To select these $m$-tuples of gaussians we shall first choose gaussians that are suitably near-extremising, and then scale them appropriately.
We pause to observe that the family of gaussian $\delta$-near extremisers, for a given datum $(\datum{L},\datum{p})$, is invariant under common isotropic scalings of their associated quadratic forms.
To be precise, the family of all positive-definite $\datum{A}=(A_j)$ satisfying
\[
\BLg(\datum{L},\datum{p};\datum{A}) \geq (1-\delta)\BL(\datum{L},\datum{p})
\]
is invariant under the scaling $\datum{A}\mapsto\lambda\datum{A}$ for all $\lambda>0$.
Using the elementary fact that
\begin{equation}\label{intgau}
\BL(\datum{L},\datum{p};\datum{f})=\BLg(\datum{L},\datum{p};\datum{A})
\end{equation}
for $\datum{f}$ satisfying \eqref{gau},
we are therefore at liberty to rescale our gaussian near-extremisers without penalty.

By Corollary \ref{effliecor}, there exists a constant $\nu \simeqone$ such that for each $u \in U_\nu(0)$ and $\delta \in (0,\nu)$, there exists an  $m$-tuple of positive definite matrices $(A_{u,\delta,j})_{j=1}^m$ satisfying
\begin{align}
    \label{e:mineig}     \|A_{u,\delta,j}^{-1}\| & < \delta^{-2\tau N}, \\
    \label{e:maxeig}      \|A_{u,\delta,j} \| & \lesssim 1, \\
    \label{brian}  \BLg(\differential\datum{B}(u),\datum{p};\datum{A}_{u,\delta}) & \geq (1-\delta^\tau) \BL(\differential\datum{B}(u),\datum{p}).
\end{align}
Here, merely for convenience, we have made a renormalisation to obtain \eqref{e:mineig} and \eqref{e:maxeig}, rather than the symmetric bounds which arise directly from Corollary \ref{effliecor}.
These bounds formalise our requirement that the $A_{u,\delta,j}$ are sufficiently ``close to the identity", since $\tau N<\gamma$, and $\gamma$ is chosen to be sufficiently small --- see \eqref{taucond} and \eqref{satisfied?!}.

In what follows the $n\times n$ matrix
\[
M_{u,\delta} := \sum_{j=1}^m p_j dB_j(u)^* A_{u,\delta,j} dB_j(u)
\]
will play a crucial role.
It naturally arises via the identity
\begin{align}\label{maingaus}
\prod_{j=1}^m \exp(-\pi p_j\langle A_{u,\delta,j}dB_j(u)x,dB_j(u)x \rangle)
&= \exp(-\pi \langle  M_{u,\delta} x,x \rangle).
\end{align}
As we shall see, the family of gaussians
\[
\exp(-\pi \langle  M_{u,\delta} x,x \rangle)
\]
(after rescaling and  suitable truncations) will be our main tool in passing from scale $\delta^\alpha$ to scale $\delta$ in Proposition \ref{prop:recursive}.
It will be important to observe that $M_{u,\delta}$ is also, in some sense, close to the identity matrix.
In particular,
there exists a constant $\nu \simeqone$ such that
\begin{equation} \label{steve}
\sup_{u\in U_\nu(0)}\|M_{u,\delta}^{-1/2}\| \lesssim \delta^{-\tau N}
\end{equation}
for all $\delta \in (0,\nu)$.
To see this, note that if $\lambda$ is any eigenvector of $M_{u,\delta}$ with corresponding unit eigenvector $x$ in $\R^n$, then
\begin{align*}
\lambda = \sum_{j=1}^m p_j \langle A_jdB_j(u)x,dB_j(u)x\rangle \gtrsim \delta^{2\tau N} \sum_{j=1}^m |dB_j(u)x|^2 \gtrsim \delta^{2\tau N}
\end{align*}
for a suitably small constant  $\nu \simeqone$, $u \in U_\nu(0)$ and $\delta \in (0,\nu)$.
Here, the first lower bound follows from \eqref{e:mineig}.
The second lower bound at the origin follows from the fact that finiteness of $\BL(\datum{L},\datum{p})$ ensures that the intersection of the kernels of the $L_j$ is trivial, and a suitably small choice of the constant $\nu \simeqone$ ensures that such a bound extends to all $u \in U_\nu(0)$.
The estimate in \eqref{steve} now follows.

Finally, we  define a family of $L^1$-normalised, scaled gaussians, denoted  $\datum{g}_{u,\delta}=(g_{u,\delta,j})_{j=1}^m$, by
\[ g_{u,\delta,j}(x)=(\gausscl)^{-n_j}\det(A_{u,\delta,j})^{1/2}\exp\left(-\pi \left\langle  A_{u,\delta,j}\frac{x}{\gausscl},\frac{x}{\gausscl} \right\rangle\right).\]
By \eqref{intgau}, \eqref{brian} and the scale-invariance of near-extremisers,
\begin{equation}\label{gnear}
\BL(\differential\datum{B}(u),\datum{p};\datum{g}_{u,\delta})  \geq (1-\delta^\tau) \BL(\differential\datum{B}(u),\datum{p}).
\end{equation}
These gaussians will be used to achieve localisation in the proof of Proposition \ref{prop:recursive}.



We conclude this section with three technical lemmas concerning the gaussians  $\datum{g}_{u,\delta}$, which feed into the proof of Proposition \ref{prop:recursive} in the next section.
The first of the lemmas essentially states that we may truncate our gaussian near-extremisers to certain $\delta$-balls without losing too much. 
\begin{lemma}\label{lem:truncation}
For each $\eta>0$ there exists a constant $\nu>0$ such that
\begin{align*}
\int_{\R^n}\prod_{j=1}^m g_{u,\delta,j}^{p_j}(dB_{j}(u) x) \,dx
\leq (1+\delta^{\eta}) \int_{U_\delta(0)}\prod_{j=1}^m g_{u,\delta,j}^{p_j}(dB_{j}(u) x) \,dx,
\end{align*}
for all $u\in U_\nu(0)$ and $\delta \in (0,\nu)$.
\end{lemma}

\begin{proof}
Fix $\eta>0$.
By elementary considerations it suffices to show that
\begin{equation} \label{e:chopgaus}
\int_{\R^n\backslash U_\delta(0)}\prod_{j=1}^m g_{u,\delta,j}^{p_j} (dB_{j}(u) x) \,dx
\lesssim \delta^{2\eta}\int_{\R^n}\prod_{j=1}^m g_{u,\delta,j}^{p_j}(dB_{j}(u) x) \,dx
\end{equation}
for sufficiently small $\delta > 0$, depending on $\eta$.
To establish \eqref{e:chopgaus} we first observe that, by \eqref{maingaus}, the integrand is expressible as
\[
\prod_{j=1}^m g_{u,\delta,j}^{p_j}(dB_{j}(u)x)
= (\gausscl)^{-n}\prod_{j=1}^m (\det A_{u,\delta,j})^{p_j/2} \exp\left(-\pi\left\langle M_{u,\delta} \frac{x}{\gausscl},\frac{x}{\gausscl}\right\rangle\right).
\]
After rescaling  and a change of variables \eqref{e:chopgaus} is equivalent to
\[
\int_{| M_{u,\delta}^{-1/2}x| \geq \delta^{-\gamma} }e^{-\pi|x|^2} \,dx \lesssim \delta^{2\eta}
\]
and, since $| M_{u,\delta}^{-1/2}x| \leq \delta^{-\tau N} |x|$ (for an appropriate choice of the constant $\nu \simeqone$) from \eqref{steve}, it suffices to show that
\begin{equation} \label{e:chopgausETS}
\int_{|x|\gtrsim \delta^{-\gamma+\tau N}}e^{-\pi|x|^2} \,dx\lesssim \delta^{2\eta}
\end{equation}
for sufficiently small $\delta$.
To see this, we simply observe that
\[
\int_{|x|\gtrsim \delta^{-\gamma+\tau N} }e^{-\pi|x|^2} \,dx\lesssim \int_{|x|\sim\delta^{-\gamma+\tau N} }e^{-\pi|x|^2} \,dx\lesssim (\delta^{-\gamma+\tau N} )^ne^{-c(\delta^{-\gamma+\tau N} )^2}
\]
for some constant $c \simeqone$, and therefore \eqref{e:chopgausETS} holds, since $\tau < \gamma/N$, provided that $\delta$ is sufficiently small.
\end{proof}

Our next technical lemma is the following, which captures the fact that gaussian near-extremisers are
stable under perturbations of their centres by amounts which are small relative to their scale.
This lemma will come into play as we transition between scales in the proof of the recursive inequality \eqref{inductionstep}, and thus refers to the more localised gaussians $g_{u,\delta^\alpha,j}$.
Before providing the precise statement, we introduce the affine mappings $L_j^{u,y} : \R^n \to \R^{n_j}$ given by
\[
L^{u,y}_jx = L^u_jx - B_j(y) = B_j(u)+dB_j(u)(x-u) - B_j(y).
\]

\begin{lemma}\label{lem:perturb} There exists a constant $\nu \simeqone$ such that if $\delta \in (0,\nu)$, then for $u\in U_{\nu}(0)$ and $y\in U_\delta(u)$,
\begin{align*}
\int_{U_{\delta^\alpha}(y)}\prod_{j=1}^m g_{u,\delta^\alpha,j}^{p_j} (dB_j(u) (x-y)) \,dx
\leq (1+\delta^{\gap}) \int_{U_{\delta^\alpha}(y)}\prod_{j=1}^m g_{u,\delta^\alpha,j}^{p_j} ( L^{u,y}_j x) \,dx.
\end{align*}
\end{lemma}

\begin{proof} Let $\nu\simeqone$ be a small constant to be specified below.
Fix $\delta \in (0,\nu)$, $u\in U_{\nu}(0)$, and $y\in U_\delta(u)$.
It then suffices to prove, for some $\gap'$ satisfying $\gap < \gap'$, that
\begin{equation}\label{beforebootstraping}
\int_{U_{\delta^\alpha}(y)}\prod_{j=1}^m g_{u,\delta^\alpha,j}^{p_j} (dB_j(u)(x-y)) \,dx - \delta^{\gap'}
\leq  \int_{U_{\delta^\alpha}(y)}\prod_{j=1}^m g_{u,\delta^\alpha,j}^{p_j} (L^{u,y}_j x) \,dx,
\end{equation}
as we may then obtain the desired estimate by a bootstrapping argument.
Indeed, by continuity of the Brascamp--Lieb constant \cite{BBCF}, \eqref{brian} and Lemma \ref{lem:truncation},
\[
{\BL(\differential\datum{B}(0), \datum{p})}
\leq 2{\BL(\differential\datum{B}(u), \datum{p})}
 \leq 4\int_{U_{\delta^\alpha}(y)}\prod_{j=1}^m g_{u,\delta^\alpha,j}^{p_j} (dB_j(u)(x-y)) \,dx,
\]
provided that $\nu$ is sufficiently small.
Therefore \eqref{beforebootstraping} implies that
\[
\left(1-\frac{4\delta^{\gap'}}{{\BL(\differential\datum{B}(0), \datum{p})}}\right)
\int_{U_{\delta^\alpha}(y)}\prod_{j=1}^m g_{u,\delta^\alpha,j}^{p_j} (dB_j(u)(x-y)) \,dx  \nonumber\\
\leq \int_{U_{\delta^\alpha}(y)}\prod_{j=1}^m g_{u,\delta^\alpha,j}^{p_j} (L^{u,y}_j(x)) \,dx,
\]
and the result quickly follows, provided that the constant $\nu \simeqone$ is sufficiently small.

In order to show \eqref{beforebootstraping}, we define
\[
G(w) = \prod_{j=1}^m g_{u,\delta^\alpha,j}^{p_j}(w_j)
\]
for $w = (w_1, \dots, w_m) \in \R^{n_1 + \dots + n_m}$.
Here, and for the remainder of the proof, we suppress the dependence on $u$ and $y$.
Thus, \eqref{beforebootstraping} follows once we show
\begin{equation} \label{e:perturbETS}
\|G \circ L -  G \circ \widetilde{L}\|_{L^1(U_{\delta^\alpha}(y))} \leq \delta^{\gap'},
\end{equation}
where $L_j = L^{u,y}_j$ and $\widetilde{L}_j = dB_j(u)(\cdot-y)$, and we shall show this using an appropriate uniform bound on the integrand.
For this, first we observe the estimate
\[
|\nabla G(w)| \lesssim (\gausscl)^{-\alpha(n+2)} |w|,
\]
which is a consequence of \eqref{e:maxeig} and a use of the scaling condition \eqref{scal} to collect powers of $\delta$.
Since moreover
\[
|L_jx - \widetilde{L}_jx| \lesssim \cutscl^2
\]
from \eqref{taylor}, the mean value theorem implies that
\[
\|G \circ L -  G \circ \widetilde{L}\|_{L^\infty(U_{\delta^\alpha}(y))} \lesssim (\gausscl)^{-\alpha(n+2)} \cutscl^{2+\alpha}, \]
and hence
\[
\|G \circ L -  G \circ \widetilde{L}\|_{L^1(U_{\delta^\alpha}(y))}  \lesssim \delta^{2-\alpha}  \left(\delta^{-\gamma} \right)^{\alpha(2+n)}.
\]
Our choice of $\gamma$ ensures that the exponent on the right-hand side of the above estimate is strictly larger than $\beta$.
Provided the constant $\nu \simeqone$ is sufficiently small, this implies the existence of $\beta' > \beta$ such that \eqref{e:perturbETS} holds, and completes the proof.
\end{proof}

Our final technical lemma shows that appropriately \textit{truncated} gaussians enjoy a local constancy property.
We note that gaussians are not $\kappa$-constant at scale $\mu$ on $\R^n$ for any choice of $\kappa$ and $\mu$ --- an observation that reflects the absence of a fixed-time parabolic Harnack principle.
\begin{lemma}\label{lem:gaussconstant}
There exists a constant $\nu \simeqone$ with the following property:
for all $\delta \in (0,\nu)$ and $\mu>0$ satisfying the threshold condition $\delta^{\alpha+\gap}>\mu$, and for all $u \in U_{\nu}(0)$ and $x\in U_{2\delta}(u)$, the function $g_{u,\delta^\alpha,j}(L^u_jx-\cdot)$ is $\exp(\delta^\loss)$-constant at scale $\mu$ on $B_j(U_{2\delta^\alpha}(x))$.
\end{lemma}

\begin{proof}
It suffices to show that, given any constant $C \simeqone$, the function $g_{u,\delta^\alpha,j}$ is $\exp(\delta^\loss)$-constant at scale $\mu$ on $U_{C\delta^\alpha}(0)$, where $u \in U_{\nu}(0)$, for an appropriate choice of the constant $\nu \simeqone$, and $\delta \in (0,\nu)$ and $\mu>0$ satisfying the threshold condition $\delta^{\alpha+\gap}>\mu$.
Indeed, $g_{u,\delta^\alpha,j}(L^u_jx-\cdot)$ is simply a translation of $g_{u,\delta^\alpha,j}$, so to establish the desired claim, it is only necessary to check that $L_j^ux - w \in U_{C\delta^\alpha}(0)$ for some constant $C \simeqone$ whenever $w \in B_j(U_{2\delta^\alpha}(x))$.
This follows since, given such $w \in B_j(U_{2\delta^\alpha}(x))$, we write $w = B_j(y)$ for some $y \in  U_{2\delta^\alpha}(x)$ and use \eqref{taylor} to deduce that
\[
|L^u_jx - w| \lesssim |x-u|^2 + |B_j(x) - B_j(y)| \lesssim \cutscl^2  + \cutscl^\alpha \sim\cutscl^\alpha.
\]

To establish that $g_{u,\delta^\alpha,j}$ is $\exp(\delta^\loss)$-constant at scale $\mu$ on $U_{C\delta^\alpha}(0)$, we first observe that, for any $R \geq 1$, $\widetilde{\mu} \leq R$, $|x| \leq R$ and $|y| \leq \widetilde{\mu}$, the gaussian $g_{A}(x)=\det(A)^{1/2}\exp(-pi \langle  Ax,x \rangle)$ satisfies
\[
g_A(x) \leq \exp(3\pi R\|A\|\widetilde{\mu}) g_A(x+y)
\]
by elementary considerations.
Applying this fact when $A = A_{u,\delta^\alpha,j}$, $R = C\delta^{-\alpha\gamma} $ and $\widetilde{\mu} = \delta^{-\alpha(1+\gamma)}\mu$, we see that $g_{u,\delta^\alpha,j}$ is $\exp(\eta)$-constant at scale $\mu$ on $CU_{\delta^\alpha}(0)$, where
\[
\eta = 3\pi C\|A_{u,\delta^\alpha,j}\|\mu\delta^{-\alpha\gamma}\delta^{-\alpha(1+\gamma)} .
\]
In order to use the preceding observation, we should ensure $\widetilde{\mu} \leq R$; this is simply a consequence of the fact that $\alpha > 1$ and $\beta, \gamma > 0$ (and provided that the constant $\nu \simeqone$ is sufficiently small).
Moreover, under the threshold condition $\delta^{\alpha+\gap}>\mu$, and for $u \in U_{\nu}(0)$ with the constant $\nu \simeqone$ sufficiently small,
\[
\eta \lesssim \delta^{\gap-2\alpha\gamma} = \delta^\loss \delta^{\gap-2\alpha\gamma-\loss}
\]
by \eqref{e:maxeig}.
From our choice of $\sigma$, it follows that $\eta \leq \delta^\loss$ (again, provided that the constant $\nu \simeqone$ is sufficiently small).
This establishes that $g_{u,\delta^\alpha,j}$ is $\exp(\delta^\loss)$-constant at scale $\mu$ on $U_{C\delta^\alpha}(0)$, for any $u \in U_{\nu}(0)$, and thus completes the proof of the lemma.
\end{proof}

\subsection{Proof of the recursive inequality}

The main idea in our proof of Proposition \ref{prop:recursive} is that it is possible to run a nonlinear variant of the argument leading to \eqref{Ball2}.
In doing so, we first claim that it is enough to prove that there exists a constant  $\nu \simeqone$ such that if $u\in U_{\nu}(0)$ and $\delta \in (0,\nu)$ satisfies $\delta^{\alpha+\gap}>\mu$, then
\begin{equation}\label{inductionstep'}
\constantC(u,\delta,\mu,\kappa) \leq (1 + \delta^{\gap})(1+2\delta^{\alpha\tau})^2 \max_{x\in U_{2\delta}(u)} \constantC(x,\delta^\alpha,\mu,\kappa\exp(\delta^\loss)).
\end{equation}
Since $\sigma < \min\{\beta,\alpha \tau\}$, the desired inequality in \eqref{inductionstep} follows from \eqref{inductionstep'} by choosing the constant $\nu \simeqone$ sufficiently small.
Our proof of \eqref{inductionstep'} appeals to Lemmas \ref{lem:truncation} to \ref{lem:gaussconstant} in sequence, with each implicitly imposing its own smallness requirement on $\nu$.
\begin{proof}[Proof of \eqref{inductionstep'}]
We fix $u \in U_{\nu}(0)$, and assume that $\delta \in (0,\nu)$ and $\mu$ satisfy the threshold condition $\delta^{\alpha+\gap}>\mu$.
We also fix $f_j\in L^1(B_j(U_{2\delta}(u));\mu,\kappa)$ and, to shorten some forthcoming expressions, we write $F = \prod_{j=1}^m f_j^{p_j} \circ B_j$.

By \eqref{gnear},
\begin{equation}\label{brianscal}
\BL(\differential\datum{B}(u),\datum{p};\datum{g}_{u,\delta^\alpha})\geq (1-\delta^{\alpha\tau}) \BL(\differential\datum{B}(u),\datum{p}),
\end{equation}
and so, for a sufficiently small constant $\nu \simeqone$,
\[
\int_{U_\delta(u)} F(y) \,dy
\leq \frac{1+2\delta^{\tau\alpha}}{\BL(\differential\datum{B}(u), \datum{p})} \int_{U_\delta(u)} F(y) \int_{\R^n}\prod_{j=1}^m g_{u,\delta^\alpha,j}^{p_j}(dB_j(u)x)  \,dx \,dy.
\]
By Lemma \ref{lem:truncation} with $\eta=\tau$, followed by a translation change of variables in $x$,
\begin{equation*}
\int_{U_\delta(u)} F(y) \,dy
\leq \frac{(1+2\delta^{\tau\alpha})^2}{\BL(\differential\datum{B}(u), \datum{p})} \int_{U_\delta(u)}F(y) \int_{U_{\delta^\alpha}(y)}\prod_{j=1}^m g_{u,\delta^\alpha,j}^{p_j}(dB_j(u)(x-y))   \,dx \,dy.
\end{equation*}
Now we may apply Lemma \ref{lem:perturb}, and then Fubini's theorem, to deduce that
\begin{align*}
\int_{U_\delta(u)}F(y) \,dy
& \leq \frac{(1+ \delta^{\gap})(1+2\delta^{\tau\alpha})^2}{\BL(\differential\datum{ B}(u), \datum{p})}  \int_{ U_\delta(u)}\int_{U_{\delta^\alpha}(y)}F(y) \prod_{j=1}^m g_{u,\delta^\alpha,j}^{p_j}(L^{u,y}_jx)  \,dx \,dy \\
& \leq \frac{(1+ \delta^{\gap})(1+2\delta^{\tau\alpha})^2}{\BL(\differential\datum{B}(u), \datum{p})}  \int_{U_{\delta}(u)+U_{\delta^\alpha}(0)} \int_{U_{\delta^\alpha}(x)} F(y) \prod_{j=1}^m g_{u,\delta^\alpha,j}^{p_j}(L^{u,y}_jx)  \,dy \,dx.
\end{align*}
We recall that $L^{u,y}_j$ is the affine map given by
\[
L^{u,y}_jx = L^u_jx - B_j(y) = B_j(u)+dB_j(u)(x-u) - B_j(y).
\]
The inner integral is thus
\[
\int_{U_{\delta^\alpha}(x)}\prod_{j=1}^m \left(h_j^x\right)^{p_j} (B_j(y)) \,dy
\]
where
\[
h_j^x(w) = f_j(w) g_{u,\delta^\alpha,j}(L^u_jx - w).
\]
By assumption, $f_j\in L^1(B_j(U_{2\delta}(u));\mu,\kappa)$.
Further, if $x\in U_\delta(u)+U_{\delta^\alpha}(0)$, then $U_{2\delta^\alpha}(x) \subseteq U_{2\delta}(u)$, whence (after restriction) $f_j\in L^1(B_j(U_{2\delta^\alpha}(x));\mu,\kappa)$.
Moreover, from Lemma \ref{lem:gaussconstant},
\[
g_{u,\delta^\alpha,j}(L^u_jx-\cdot)\in L^1(B_j(U_{2\delta^\alpha}(x));\mu,\exp(\delta^\loss)),
\]
and therefore
\[
h_j^x\in L^1(B_j(U_{2\delta^\alpha}(x));\mu,\kappa\exp(\delta^\loss))
\]
whenever $x\in U_\delta(u)+U_{\delta^\alpha}(0)$.

Using the definition of $\constantC(x,\delta^\alpha,\mu,\kappa\exp(\delta^\loss))$, and slightly enlarging the domain of integration in $x$, we see that
\begin{align*}
\int_{U_\delta(u)}F(x) \,dx
&\leq \frac{(1+ \delta^{\gap})(1+2\delta^{\tau\alpha})^2\constantM}{\BL(\differential\datum{ B}(u), \datum{p})} \int_{U_{2\delta}(u)} \prod_{j=1}^m \left(\int_{\R^{n_j}} h_j^x(y) \,dy\right)^{p_j} \,dx \\
&= \frac{(1+ \delta^{\gap})(1+2\delta^{\tau\alpha})^2 \constantM}{\BL(\differential\datum{B}(u), \datum{p})} \int_{U_{2\delta}(u)} \prod_{j=1}^m (f_j*g_{u,\delta^\alpha,j})^{p_j} (L^u_j x) \,dx,
\end{align*}
where $\constantM = \max_{x \in U_{2\delta}(u)} \constantC(x,\delta^\alpha,\mu,\kappa\exp(\delta^\loss))$.
To conclude, we simply apply the linear inequality and the fact that the gaussians $g_{u,\delta^\alpha,j}$ are normalised in $L^1$ to obtain
\begin{align*}
\int_{U_\delta(u)} F(x) \,dx
&\leq (1+ \delta^{\gap})(1+2\delta^{\tau\alpha})^2\constantM \prod_{j=1}^m \left(\int_{\R^{n_j}}f_j*g_{u,\delta^\alpha,j}(x+B_j(u)) \,dx\right)^{p_j} \\
&=  (1+ \delta^{\gap})(1+2\delta^{\tau\alpha})^2 \constantM \prod_{j=1}^m
        \left(\int_{\R^{n_j}} f_j(x_j) \,dx_j\right)^{p_j}
\end{align*}
and \eqref{inductionstep'} follows.
\end{proof}

\end{document}